\documentclass[final]{siamart0516}
\usepackage{amsfonts}
\usepackage{graphicx}
\usepackage{epstopdf}
\usepackage{amsmath,amssymb,enumitem}
\allowdisplaybreaks
\newtheorem{example}[theorem]{Example}

\newcommand{\nprod}[1]{{\cdot}_#1}
\newcommand{\unf}[2]{{\mathbf  #1}_{(#2)}}
\newcommand{\newA}{\mathbf H}
\newcommand{\newC}{\mathbf G}
\newcommand{\newc}{\mathbf g}
\newcommand{\newK}{I_3}
\newcommand{\newn}{I_1}
\newcommand{\DrawSetOfSolutions}[4]{
\begin{tikzpicture}[line join=bevel,x={(#4*0.25\textwidth,#4*0.375\textwidth)},y={(#4*\textwidth,0)},z={(0,#4*\textwidth)}]
	\tikzset{%
	    my thick dashed color/.style={thick,color=blue,dashed},
	    my thick solid color/.style={thick,color=blue},
	    my thin solid color/.style={thin,color=gray},
	    my thin dashed color/.style={thin,color=gray,dashed},
	    vertex color/.style={color=red}
	}

\def\II{1/#1}
\def\JJ{1/#2}
\def\KK{1/#3}

\coordinate (S) at (\II,\JJ,\KK);
\coordinate (S2) at (\II,\II,\II);
\coordinate (N) at (1,1,1);
\coordinate (X1) at (1-\JJ+\KK,\JJ,\KK);
\coordinate (X2) at (1,\JJ,\JJ);
\coordinate (Y1) at (\II,1-\II+\KK,\KK);
\coordinate (Y2) at (\II,1,\II);
\coordinate (Z1) at (\II,\JJ,1-\II+\JJ);
\coordinate (Z2) at (\II,\II,1);

\draw [my thick solid color] (S) -- (Z1) -- (Z2) -- (Y2) -- (Y1) -- cycle;
\draw [my thick solid color] (N) -- (Z2);
\draw [my thick solid color] (N) -- (Y2);

\draw [my thick dashed color] (Z1) -- (X2) -- (X1) -- (S);
\draw [my thick dashed color] (N) -- (X2);
\draw [my thick dashed color] (X1) -- (Y1);

\coordinate (P1) at (\II,\JJ,1);
\coordinate (P2) at (1,\JJ,1);
\coordinate (P4) at (\II,1,1);
\coordinate (Q2) at (1,\JJ,\KK);
\coordinate (Q3) at (1,1,\KK);
\coordinate (Q4) at (\II,1,\KK);

\draw [my thin solid color] (Z1) -- (P1) -- (Z2) -- (P4) -- (Y2) -- (Q4) -- (Y1);
\draw [my thin solid color] (P1) -- (P2) -- (N) -- (P4);
\draw [my thin solid color] (N) -- (Q3) -- (Q4);
\draw [my thin dashed color] (P2) -- (X2) -- (Q2) -- (X1);
\draw [my thin dashed color] (Q2) -- (Q3);

\fill [vertex color] (S)   circle[radius=2pt];
\ifthenelse{#1=#2 \AND #2=#3} %
{ 
\draw (S) node  [below left] {$S$}} %
{ 
\draw (S) node  [below left] {$S$}};
\fill [vertex color] (N)  circle[radius=2pt];
\draw (N) node  [above right] {$N$};
\ifthenelse{#2=#3} %
{\fill [vertex color] (X1)  circle[radius=2pt];
 \draw (X1) node  [left] {${X}$}} %
{\fill [vertex color] (X1)  circle[radius=2pt];
\draw (X1) node  [left] {${X_1}$};
\fill [vertex color] (X2)  circle[radius=2pt];
\draw (X2) node  [left] {${X_2}$}};
\ifthenelse{#1=#3} %
{\fill [vertex color] (Y1)  circle[radius=2pt];
\draw (Y1) node  [below right] {${Y}$}} %
{\fill [vertex color] (Y1)  circle[radius=2pt];
\draw (Y1) node  [below] {${Y_1}$};
\fill [vertex color] (Y2)  circle[radius=2pt];
\draw (Y2) node  [right] {${Y_2}$}};
\ifthenelse{#1=#2} %
{\fill [vertex color] (Z1)  circle[radius=2pt];
\draw (Z1) node  [above left] {${Z}$}} %
{\fill [vertex color] (Z1)  circle[radius=2pt];
\draw (Z1) node  [left] {${Z_1}$};
\fill [vertex color] (Z2)  circle[radius=2pt];
\draw (Z2) node  [above] {${Z_2}$}};

\coordinate (C) at (0,0,0);
    \coordinate (E1) at (1,0,0);
    \coordinate (E2) at (0,1,0);
    \coordinate (E3) at (0,0,1);
    \draw[->] (C) -- ($(C)+(E1)$) node  [left] {${\sigma_1^2}$};
    \draw[->] (C) -- ($(C)+(E2)$) node  [below] {${\sigma_2^2}$};
    \draw[->] (C) -- ($(C)+(E3)$) node  [left] {${\sigma_3^2}$};
    \draw (C) node  [below left] {${O}$};

\end{tikzpicture}
}
\newcommand{\Drawdifference}[4]{
\begin{tikzpicture}[line join=bevel,x={(#4*0.25\textwidth,#4*0.375\textwidth)},y={(#4*\textwidth,0)},z={(0,#4*\textwidth)}]

	\tikzset{%
	    my thick dashed color/.style={thick,color=blue,dashed},
	    my thick solid color/.style={thick,color=blue},
	    my thin solid color/.style={thin,color=gray},
	    my thin dashed color/.style={thin,color=gray,dashed},
	    vertex color/.style={color=red}
	}

\def\II{1/#1}
\def\JJ{1/#2}
\def\KK{1/#3}

\coordinate (S) at (\II,\JJ,\KK);
\coordinate (S2) at (\II,\II,\II);
\coordinate (N) at (1,1,1);
\coordinate (X1) at (1-\JJ+\KK,\JJ,\KK);
\coordinate (X2) at (1,\JJ,\JJ);
\coordinate (Y1) at (\II,1-\II+\KK,\KK);
\coordinate (Y2) at (\II,1,\II);
\coordinate (Z1) at (\II,\JJ,1-\II+\JJ);
\coordinate (Z2) at (\II,\II,1);

\draw [my thick solid color] (S) -- (Z1) -- (Z2) -- (S2) -- (Y2) -- (Y1) -- cycle;

\draw [my thick dashed color] (Z1) -- (X2) -- (X1) -- (S);
\draw [my thick dashed color] (X1) -- (Y1);
\draw [my thick dashed color] (X2) -- (S2);
\draw [my thick solid color] (Y2) -- (S2);
\draw [my thick solid color] (Z2) -- (S2);
\draw [my thick dashed color] (X2) -- (Y2);
\draw [my thick dashed color] (X2) -- (Z2);

\coordinate (P1) at (\II,\JJ,1);
\coordinate (P2) at (1,\JJ,1);
\coordinate (P4) at (\II,1,1);
\coordinate (Q2) at (1,\JJ,\KK);
\coordinate (Q3) at (1,1,\KK);
\coordinate (Q4) at (\II,1,\KK);

\draw [my thin solid color] (Z1) -- (P1) -- (Z2) -- (P4) -- (Y2) -- (Q4) -- (Y1);
\draw [my thin solid color] (P1) -- (P2) -- (N) -- (P4);
\draw [my thin solid color] (N) -- (Q3) -- (Q4);
\draw [my thin dashed color] (P2) -- (X2) -- (Q2) -- (X1);
\draw [my thin dashed color] (Q2) -- (Q3);

\fill [vertex color] (S2)   circle[radius=2pt];
\draw (S2) node  [above right] {$S_2$};
\fill [vertex color] (S)   circle[radius=2pt];
\ifthenelse{#1=#2 \AND #2=#3} %
{ 
\draw (S) node  [below left] {$S$}} %
{ 
\draw (S) node  [below left] {$S$}};
\fill [vertex color] (N)  circle[radius=2pt];
\draw (N) node  [above right] {$N$};
\ifthenelse{#2=#3} %
{\fill [vertex color] (X1)  circle[radius=2pt];
 \draw (X1) node  [left] {${X}$}} %
{\fill [vertex color] (X1)  circle[radius=2pt];
\draw (X1) node  [left] {${X_1}$};
\fill [vertex color] (X2)  circle[radius=2pt];
\draw (X2) node  [left] {${X_2}$}};
\ifthenelse{#1=#3} %
{\fill [vertex color] (Y1)  circle[radius=2pt];
\draw (Y1) node  [below right] {${Y}$}} %
{\fill [vertex color] (Y1)  circle[radius=2pt];
\draw (Y1) node  [below] {${Y_1}$};
\fill [vertex color] (Y2)  circle[radius=2pt];
\draw (Y2) node  [right] {${Y_2}$}};
\ifthenelse{#1=#2} %
{\fill [vertex color] (Z1)  circle[radius=2pt];
\draw (Z1) node  [above left] {${Z}$}} %
{\fill [vertex color] (Z1)  circle[radius=2pt];
\draw (Z1) node  [left] {${Z_1}$};
\fill [vertex color] (Z2)  circle[radius=2pt];
\draw (Z2) node  [above] {${Z_2}$}};

\coordinate (C) at (0,0,0);
    \coordinate (E1) at (1,0,0);
    \coordinate (E2) at (0,1,0);
    \coordinate (E3) at (0,0,1);
    \draw[->] (C) -- ($(C)+(E1)$) node  [left] {${\sigma_1^2}$};
    \draw[->] (C) -- ($(C)+(E2)$) node  [below] {${\sigma_2^2}$};
    \draw[->] (C) -- ($(C)+(E3)$) node  [left] {${\sigma_3^2}$};
    \draw (C) node  [below left] {${O}$};

\end{tikzpicture}
}
\newcommand{\DrawSetOfSolutionsproved}[4]{
\begin{tikzpicture}[line join=bevel,x={(#4*0.25\textwidth,#4*0.375\textwidth)},y={(#4*\textwidth,0)},z={(0,#4*\textwidth)}]
	\tikzset{%
	    my thick dashed color/.style={thick,color=blue,dashed},
	    my thick solid color/.style={thick,color=blue},
	    my thin solid color/.style={thin,color=gray},
	    my thin dashed color/.style={thin,color=gray,dashed},
	    vertex color/.style={color=red}
	}

\def\II{1/#1}
\def\JJ{1/#2}
\def\KK{1/#3}

\coordinate (S) at (\II,\JJ,\KK);
\coordinate (S2) at (\II,\II,\II);
\coordinate (N) at (1,1,1);
\coordinate (X1) at (1-\JJ+\KK,\JJ,\KK);
\coordinate (X2) at (1,\JJ,\JJ);
\coordinate (Y1) at (\II,1-\II+\KK,\KK);
\coordinate (Y2) at (\II,1,\II);
\coordinate (Z1) at (\II,\JJ,1-\II+\JJ);
\coordinate (Z2) at (\II,\II,1);

\draw [my thick solid color] (X2) -- (Z2) -- (S2) -- (Y2) -- (Z2) -- (N) -- (Y2) -- (X2) -- (S2);

\draw [my thick dashed color] (N) -- (X2);

\coordinate (P1) at (\II,\JJ,1);
\coordinate (P2) at (1,\JJ,1);
\coordinate (P4) at (\II,1,1);
\coordinate (Q2) at (1,\JJ,\KK);
\coordinate (Q3) at (1,1,\KK);
\coordinate (Q4) at (\II,1,\KK);

\draw [my thin solid color] (S) -- (P1) -- (Z2) -- (P4) -- (Y2) -- (Q4) -- cycle;
\draw [my thin solid color] (P1) -- (P2) -- (N) -- (P4);
\draw [my thin solid color] (N) -- (Q3) -- (Q4);
\draw [my thin dashed color] (P2) -- (X2) -- (Q2) -- (S);
\draw [my thin dashed color] (Q2) -- (Q3);

\fill [vertex color] (N)  circle[radius=2pt];
\draw (N) node  [above right] {$N$};
\fill [vertex color] (X2)  circle[radius=2pt];
\draw (X2) node  [left] {${X_2}$};
\fill [vertex color] (Y2)  circle[radius=2pt];
\draw (Y2) node  [right] {${Y_2}$};
\fill [vertex color] (Z2)  circle[radius=2pt];
\draw (Z2) node  [above] {${Z_2}$};
\fill [vertex color] (S2)   circle[radius=2pt];
\draw (S2) node  [above right] {$S_2$};

\coordinate (C) at (0,0,0);
    \coordinate (E1) at (1,0,0);
    \coordinate (E2) at (0,1,0);
    \coordinate (E3) at (0,0,1);
    \draw[->] (C) -- ($(C)+(E1)$) node  [left] {${\sigma_1^2}$};
    \draw[->] (C) -- ($(C)+(E2)$) node  [below] {${\sigma_2^2}$};
    \draw[->] (C) -- ($(C)+(E3)$) node  [left] {${\sigma_3^2}$};
    \draw (C) node  [below left] {${O}$};

\end{tikzpicture}
}
\usepackage{tikz}
\usetikzlibrary{calc}
\ifpdf
  \DeclareGraphicsExtensions{.eps,.pdf,.png,.jpg}
\else
  \DeclareGraphicsExtensions{.eps}
\fi

\newcommand{\TheTitle}{On the largest multilinear singular values of higher-order tensors} 
\newcommand{\TheAuthors}{Ignat Domanov, Alwin Stegeman, and Lieven De Lathauwer}

\headers{\TheTitle}{\TheAuthors}

\title{{\TheTitle}\thanks{Submitted to the editors DATE.
\funding{This work was funded by (1) Research Council KU Leuven: C1 project c16/15/059-nD; (2) F.W.O.:  project  G.0830.14N, G.0881.14N;  (3) the Belgian Federal Science Policy Office: IUAP P7 (DYSCO II,  Dynamical systems, control
 and optimization,  2012-2017); (4)  EU: The research leading to these results has received funding from the European Research Council under the European Union's Seventh Framework Programme (FP7/2007-2013) / ERC Advanced Grant: BIOTENSORS (no.  339804). This paper reflects only the authors' views and the Union is not liable for any use that may be made of the contained information}}}

\author{
  Ignat Domanov\thanks{
  Group Science, Engineering and Technology, KU Leuven - Kulak,
  E. Sabbelaan 53, 8500 Kortrijk, Belgium and
  Dept. of Electrical Engineering  ESAT/STADIUS KU Leuven,
  Kasteelpark Arenberg 10, bus 2446, B-3001 Leuven-Heverlee, Belgium
  (\email{ignat.domanov@kuleuven.be}, \email{alwin.stegeman@kuleuven.be}, \email{lieven.delathauwer@kuleuven.be}).}
  \and
  Alwin Stegeman\footnotemark[2]
  \and
  Lieven De Lathauwer\footnotemark[2]
    }

\usepackage{amsopn}

\ifpdf
\hypersetup{
  pdftitle={\TheTitle},
  pdfauthor={\TheAuthors}
}
\fi

\usepackage[caption=false]{subfig}
\usepackage{cite}
\usepackage{ifthen}

\begin{document}

\maketitle

\begin{abstract}
  Let $\sigma_n$ denote  the largest mode-$n$ multilinear singular value of an $I_1\times\dots \times I_N$  tensor $\mathcal T$.
  We prove that
  \begin{equation}
      \sigma_1^2+\dots+\sigma_{n-1}^2+\sigma_{n+1}^2+\dots+\sigma_{N}^2\leq (N-2)\|\mathcal T\|^2 + \sigma_n^2,\quad n=1,\dots,N,
      \label{eq:theveryfirsteq}
  \end{equation}
where $\|\cdot\|$ denotes the Frobenius norm.
  We also show that at least  for the  cubic tensors the inverse problem always has a solution. Namely, for each
     $\sigma_1,\dots,\sigma_N$ that satisfy \eqref{eq:theveryfirsteq}
       and the trivial inequalities $\sigma_1\geq \frac{1}{\sqrt{I}}\|\mathcal T\|,\dots, \sigma_N\geq \frac{1}{\sqrt{I}}\|\mathcal T\|$,  there always exists an $I\times \dots\times I$ tensor whose largest multilinear singular values are      equal to $\sigma_1,\dots,\sigma_N$.  For $N=3$ we
      show that  if the equality $\sigma_1^2+\sigma_2^2= \|\mathcal T\|^2 + \sigma_3^2$ in \cref{eq:theveryfirsteq} holds, then 
    $\mathcal T$ is necessarily  equal to a sum of  multilinear rank-$(L_1,1,L_1)$  and multilinear rank-$(1,L_2,L_2)$ tensors and we  give a complete description of all its multilinear singular values.  
    We establish a connection with honeycombs and eigenvalues of the sum of two Hermitian matrices. This seems to give at least a partial explanation of why results on the joint distribution of multilinear singular values are scarce.
\end{abstract}



\section{Introduction}
Throughout the paper $\|\cdot\|$ denotes the Frobenius norm of a vector, matrix, or tensor and the  superscripts $^T$, $^H$, and $^*$ denote transpose, hermitian transpose, and conjugation, respectively.
We also use the ``empty sum/product'' convention, i.e., if $m>n$, then $\sum\limits_{m}^n(\cdot) =0$ and  
$\prod\limits_{m}^n(\cdot) =1$.

Let $\mathcal T\in\mathbb C^{I_1\times\dots\times I_N}$. A {\em mode-$n$ fiber} of $\mathcal T$ is a column vector obtained by fixing indices $i_1,\dots,i_{n-1},i_{n+1},\dots,i_N$. A matrix $\unf{T}{n}\in\mathbb C^{I_n\times I_1\cdots I_{n-1}I_{n+1}\dots I_N}$ formed by all mode-$n$ fibers is called a {\em mode-$n$
matrix unfolding} (aka flattening or matricization) of  $\mathcal T$. For notational convenience  we  assume  that the
columns of $\unf{T}{n}$ are ordered  such that 
\begin{equation}\label{eq:mode_n-unfolding}
\text{the } (i_n,1+\sum_{\substack{k=1\\ k\ne n}}^N (i_k-1)\prod_{\substack{l=1\\ l\ne n}}^{k-1}I_l)\text{th entry of }\unf{T}{n} = \text{ the }(i_1,\dots, i_N)\text{th entry of }\mathcal T. 
\end{equation}
For instance, if $N=3$, i.e.,  $\mathcal T\in\mathbb C^{I_1\times I_2\times I_3}$, then
\eqref{eq:mode_n-unfolding} implies that
\begin{align*}
\unf{T}{1} &= [\mathbf T_1\ \dots\ \mathbf T_{I_3}]\in\mathbb C^{I_1\times I_2I_3},\\
\unf{T}{2} &= [\mathbf T_1^T\ \dots\ \mathbf T_{I_3}^T]\in\mathbb C^{I_2\times I_1I_3},\\
\unf{T}{3} &= [\operatorname{vec}(\mathbf T_1)\ \dots \operatorname{vec}(\mathbf T_{I_3})]^T\in\mathbb C^{I_3\times I_1I_2},
\end{align*}
where
$\mathbf T_1,\dots, \mathbf T_{I_3}\in\mathbb C^{I_1\times I_2}$ denote the frontal slices of $\mathcal T$.

Tensor $\mathcal T\in\mathbb C^{I_1\times\dots\times I_N}$ is {\em all-orthogonal} if the matrices $\unf{T}{1}\unf{T}{1}^H,\dots,\unf{T}{N}\unf{T}{N}^H$ are diagonal.
{\em The MultiLinear (ML) Singular Value Decomposition (SVD)} (aka Higher-Order SVD) is a factorization of  $\mathcal T$ into  the product of an all-orthogonal tensor $\mathcal S\in \mathbb C^{I_1\times\dots\times I_N}$ and $N$  unitary matrices $\mathbf U_1\in\mathbb C^{I_1\times I_1},\dots,\mathbf U_N\in\mathbb C^{I_N\times I_N}$,
\begin{equation}
\mathcal T = \mathcal S\nprod{1}\mathbf U_1\nprod{2}\mathbf U_2\dots \nprod{N}\mathbf U_N,\label{eq:HOSVD}
\end{equation}
where ''$\nprod{n}$'' denotes the $n$-mode product of $\mathcal S$ and $\mathbf U_n$.
Rather than giving the formal definition of ''$\nprod{n}$'', for which we refer the reader to \cite{IEEEreview, Koldareview,HOSVD}, we present $N$ equivalent matricized  versions of \eqref{eq:HOSVD}: 
\begin{equation}\label{eq:HOSVDmatrixform}
\unf{T}{n} = \mathbf U_n\unf{S}{n}(\mathbf U_N\otimes\dots\otimes\mathbf U_{n+1}\otimes\mathbf U_{n-1}\otimes\dots\otimes\mathbf U_1)^T,\qquad n=1,\dots,N,
\end{equation}
where ``$\otimes$'' denotes the Kronecker product.
 For $N=2$, i.e., for $\mathcal T=\mathbf T_1\in\mathbb C^{I_1\times I_2}$, the MLSVD  reduces, up to trivial indeterminacies, to the classical SVD of a matrix, $\unf{T}{1}=\mathbf T_1 = \mathbf U\mathbf S\mathbf V^H$, where $\mathbf U=\mathbf U_1$, $\mathbf S=\unf{S}{1}$, and $\mathbf V=\mathbf U_2^*\otimes 1$. 
 It is known \cite{HOSVD} that MLSVD always exists and that its  uniqueness properties are similar to those of the matrix SVD.
 
 The MLSVD has many applications in signal processing, data analysis, and machine learning 
(see, for instance, the overview papers \cite[Subsection 4.4]{Koldareview},\cite{review_Nikos_2016}). 
Here we just mention that as Principal Component Analysis (PCA)  can be done by SVD of a data matrix, MLPCA   can be done by  MLSVD of a data tensor \cite{Tucker1966,Kroonenbergbook,LievenMLPCA}.
 
The singular values of $\unf{T}{n}$, 
are called {\em the mode-$n$ singular values}  of $\mathcal T$. 
Since $\unf{S}{1}\unf{S}{1}^H, \dots,\unf{S}{N}\unf{S}{N}^H$ are diagonal, it follows from 
\eqref{eq:HOSVDmatrixform} that the ML singular values  of $\mathcal T$ coincide with the ML singular values of $\mathcal S$,
which are just the Frobenius norms of the rows of $\unf{S}{1},\dots,\unf{S}{N}$. Throughout  the paper,
$$
\sigma_n\ \text{ denotes the largest singular value of }\unf{T}{n}.
$$
In the matrix case, i.e., for $N=2$, the description of MLSVD is trivial.  Indeed, the singular values of $\unf{T}{1}=\mathbf T_1$ and $\unf{T}{2}=\mathbf T_1^T$ coincide and $\unf{T}{3}=\operatorname{vec}(\mathbf T_1)^T$ has a single singular value $\|\mathcal T\|$. 
Thus, the singular values of $\unf{T}{1}$ completely define the singular values of $\unf{T}{2}$ and $\unf{T}{3}$.
In particular, the set of triplets $(\sigma_1,\sigma_2,\sigma_3)$ coincides with the set $\{(x,x,y):\ y\geq x\geq 0\}\subset\mathbb R^3$ whose Lebesgue measure is zero. The situation for tensors is much more complicated.
It is clear that in the general case $N\geq 2$, the sets of  the mode-$1$,\dots, mode-$N$ singular values  are not independent either.
The study of topological properties  of the set of ML singular values of real tensors has  been initiated only recently in \cite{Uschmajev1} and \cite{Uschmajev2}. In particular, it has been shown in \cite{Uschmajev1} and \cite{Uschmajev2} that, as in the matrix case,
some configurations of ML singular values  are not possible but, nevertheless, at least for $n\times\dots\times n$ tensors
 the set of ML singular values has a positive Lebesgue measure.

In this paper we study  possible configurations for the largest ML singular values, i.e., for $\sigma_1,\dots,\sigma_N$.
Our results are valid for real and complex tensors.
The following theorem presents simple necessary conditions for $\sigma_1$, $\sigma_2$, and $\sigma_3$ to be
the largest ML singular values of a third-order tensor. For instance, it implies that a norm-$1$ tensor whose largest ML singular values are equal to $0.9$, $0.9$, and $0.7$ does not exist.   
\begin{theorem}\label{th:main}
Let 
$\sigma_1$, $\sigma_2$, and $\sigma_3$ denote the  largest  ML singular values of an $I_1\times I_2\times I_3$
tensor $\mathcal T$.
 Then 
 \begin{align}
  \sigma_1^2+\sigma_2^2\leq \|\mathcal T\|^2 + \sigma_3^2,\quad
 &\sigma_1^2+\sigma_3^2\leq \|\mathcal T\|^2 + \sigma_2^2,\quad
 \sigma_2^2+\sigma_3^2\leq \|\mathcal T\|^2 + \sigma_1^2,\label{eqeqs}\\
 \sigma_1\geq \frac{1}{\sqrt{I_1}}\|\mathcal T\|,\ 
 &\sigma_2\geq \frac{1}{\sqrt{I_2}}\|\mathcal T\|,\ 
 \sigma_3\geq \frac{1}{\sqrt{I_3}}\|\mathcal T\|.\label{eqeqstrivial}
  \end{align}
\end{theorem}
Figure \ref{figure} shows four  typical shapes of the set 
$\{(\sigma_1^2,\sigma_2^2,\sigma_3^2):\ \sigma_1,\ \sigma_2,\ \sigma_3\ \text{satisfy \eqref{eqeqs}--\eqref{eqeqstrivial}}\}$
(WLOG, we assumed  that $I_1\leq I_2\leq I_3$).
\begin{figure}[tbhp]
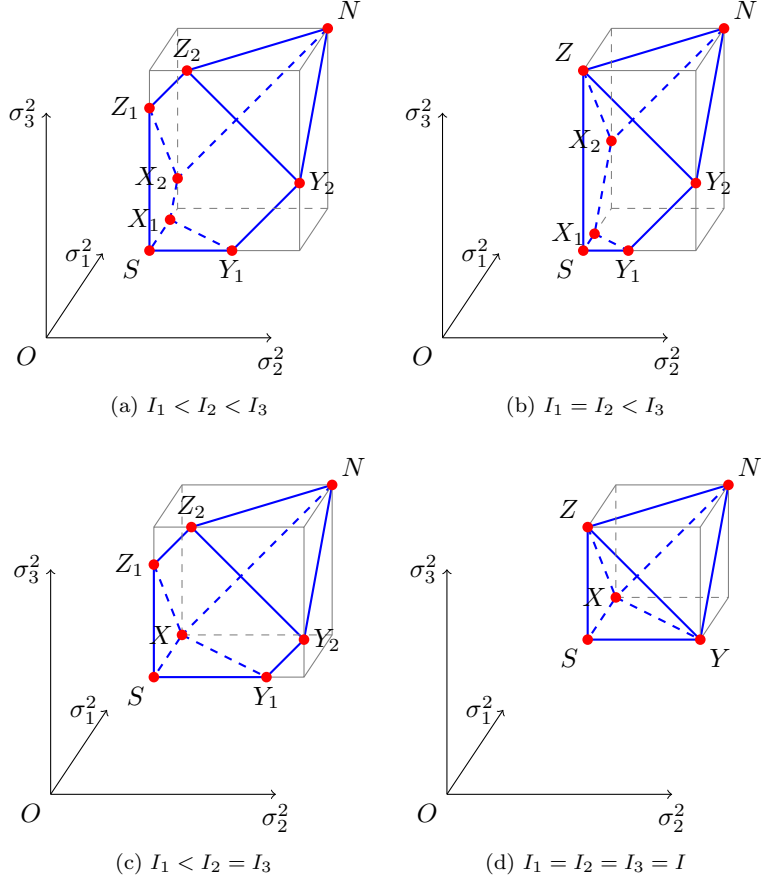

\centering
\subfloat[$I_1<I_2<I_3$]{\label{fig:a}
   \DrawSetOfSolutions{2}{3}{5}{0.23}
}
\subfloat[$I_1=I_2<I_3$]{\label{fig:b}
   \DrawSetOfSolutions{2}{2}{5}{0.23}
}\\
\subfloat[$I_1<I_2=I_3$]{\label{fig:c}
   \DrawSetOfSolutions{2}{3}{3}{0.23}
}
\subfloat[$I_1=I_2=I_3=I$]{\label{fig:d} 
\DrawSetOfSolutions{2}{2}{2}{0.23}}
\caption{The  typical shapes of the set $\{(\sigma_1^2,\sigma_2^2,\sigma_3^2):\ \sigma_1,\ \sigma_2,\ \sigma_3\ \text{satisfy \eqref{eqeqs}--\eqref{eqeqstrivial}}\}$ for $I_1\leq I_2\leq I_3$ (drawn for  $I_1=2$, $I_2=3$, $I_3=5$ and $\|\mathcal T\|=1$). Plot (a) is the case where all dimensions of a tensor are distinct. The points $S$, $X_1$, $X_2$, $Y_1$, $Y_2$, $Z_1$, $Z_2$ and $N$ have coordinates 
$(\frac{1}{I_1},\frac{1}{I_2},\frac{1}{I_3})$, 
$(1-\frac{1}{I_2}+\frac{1}{I_3},\frac{1}{I_2},\frac{1}{I_3})$, 
$(1,\frac{1}{I_2},\frac{1}{I_2})$, 
$(\frac{1}{I_1},1-\frac{1}{I_1}+\frac{1}{I_3},\frac{1}{I_3})$, 
$(\frac{1}{I_1},1,\frac{1}{I_1})$, 
$(\frac{1}{I_1},\frac{1}{I_2},1-\frac{1}{I_1}+\frac{1}{I_2})$, 
$(\frac{1}{I_1},\frac{1}{I_1},1)$, 
  and $(1,1,1)$, 
 respectively.
 Plots  (b)--(c) are the  cases where a tensor has exactly two equal dimensions, the points $Z_1$ and $Z_2$ were merged into one point $Z$  and the points $X_1$ and $X_2$ were merged into one point $X$. Plot (d) is the case where all three dimensions of a tensor are equal to each other, $I_1=I_2=I_3=I$. In this case, the points $Y_1$ and $Y_2$ were merged into one point $Y$, so
 $S$, $X$, $Y$, and $Z$ have the  coordinates $(\frac{1}{I},\frac{1}{I},\frac{1}{I})$, $(1,\frac{1}{I},\frac{1}{I})$, $(\frac{1}{I},1,\frac{1}{I})$, and $(\frac{1}{I},\frac{1}{I},1)$, respectively.
  By Corollary \ref{th:2}, any point $(\sigma_1^2,\sigma_2^2,\sigma_3^2)$ of the polyhedron $SXYZN$ in plot (d)  is feasible, i.e., there exists a norm-$1$ tensor $\mathcal T\in\mathbb C^{I\times I\times I}$ whose  squared largest multilinear singular values are $\sigma_1^2$, $\sigma_2^2$, and $\sigma_3^2$. The volume of $SXYZN$ equals half of the volume of the cube, i.e., $\frac{1}{2}(1-\frac{1}{I})^3$.
  }
\label{figure}
\end{figure}

One can easily verify that if  
$\sigma_1$, $\sigma_2$ and $\sigma_3$ satisfy \eqref{eqeqs}--\eqref{eqeqstrivial} for $I_1=I_2=I_3=2$ and $\|\mathcal T\|=1$,
then $\sigma_1$, $\sigma_2$ and $\sigma_3$ are  the largest ML singular values of the $2\times 2\times 2$ tensor $\mathcal T$ with  mode-$1$ matrix
unfolding
\begin{align*}
\unf{T}{1} &= [\mathbf T_1\ \mathbf T_2]=
\begin{bmatrix}
\frac{\sqrt{\sigma_1^2+\sigma_2^2+\sigma_3^2-1}}{\sqrt{2}}   & 0 & 0  &  \frac{\sqrt{1+\sigma_1^2-\sigma_2^2-\sigma_3^2}}{\sqrt{2}}\\
      0   & \frac{\sqrt{1+\sigma_3^2-\sigma_1^2-\sigma_2^2}}{\sqrt{2}} & \frac{\sqrt{1+\sigma_2^2-\sigma_1^2-\sigma_3^2}}{\sqrt{2}} & 0
\end{bmatrix}
.
\end{align*}
The proof of the following result relies on a similar explicit construction of an $I_1\times I_2\times I_3$ tensor $\mathcal T$.
\begin{theorem}\label{th:2nonsquare}
Let $I_1\leq I_2\leq I_3 $ and  $\sigma_1$, $\sigma_2$, $\sigma_3$ satisfy  \eqref{eqeqs} and the following three inequalities
\begin{align}
\sigma_1&\geq \frac{1}{\sqrt{I_1}}\|\mathcal T\|,\label{eq:restr1}\\
(I_2-I_1)\sigma_1^2 + (I_1I_2 -I_2)\sigma_3^2+(1-I_2)\|\mathcal T\|^2&\geq 0,\label{eq:restr2}\\
(I_2-I_1)\sigma_1^2 + (I_1I_2 -I_2)\sigma_2^2+(1-I_2)\|\mathcal T\|^2&\geq 0.\label{eq:restr3}
\end{align}
 Then there exists an 
$I_1\times I_2\times I_3$ tensor $\mathcal T$ such that
\begin{enumerate}
\item all entries of $\mathcal T$ are non-negative;
\item $\mathcal T$ is all-orthogonal;
\item the largest ML singular values of  $\mathcal T$ are equal to $\sigma_1$, $\sigma_2$ and $\sigma_3$. 
 \end{enumerate}
\end{theorem}
Conditions \eqref{eqeqstrivial} and  \eqref{eq:restr1}--\eqref{eq:restr3} mean that the point $(\sigma_1^2,\sigma_2^2,\sigma_3^2)$ belongs to the trihedral angle $SX_1Y_1Z_1$ and $S_2X_2Y_2Z_2$, respectively, where $S_2$ has coordinates $(\frac{1}{I_1},\frac{1}{I_1},\frac{1}{I_1})$.
The gap between the  necessary conditions in Theorem \ref{th:main} and the  sufficient conditions in Theorem \ref{th:2nonsquare},
i.e., the set
\begin{equation} 
\{(\sigma_1^2,\sigma_2^2,\sigma_3^2):\  \text{\eqref{eqeqs}--\eqref{eqeqstrivial} hold and  at least one of \eqref{eq:restr1}--\eqref{eq:restr3}
 does not hold}\},\label{eq:gap}
\end{equation}
is shown in Figure \ref{fig2:c}. One can easily verify that the gap is empty  only for $I_1=I_2=I_3$.
\begin{figure}[tbhp]
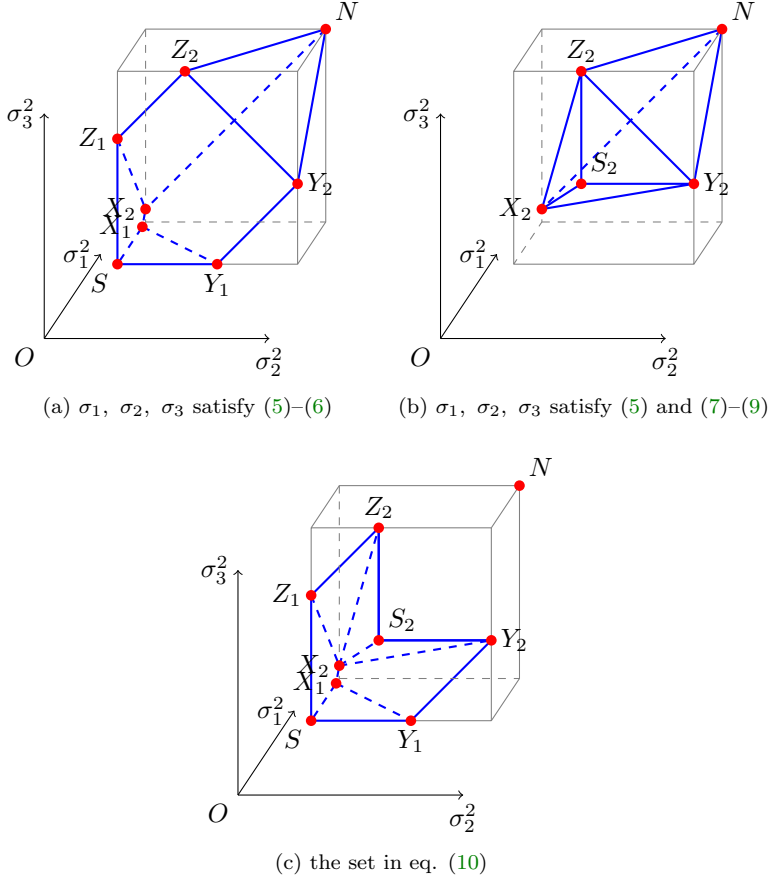

\centering
\subfloat[$\sigma_1,\ \sigma_2,\ \sigma_3\ \text{satisfy \eqref{eqeqs}--\eqref{eqeqstrivial}}$]{\label{fig2:a}
   \DrawSetOfSolutions{2}{5}{7}{0.23}
}
\subfloat[$\sigma_1,\ \sigma_2,\ \sigma_3\ \text{satisfy \eqref{eqeqs} and \eqref{eq:restr1}--\eqref{eq:restr3}}$ ]{\label{fig2:b}
   \DrawSetOfSolutionsproved{2}{5}{7}{0.23}
}
\\
\subfloat[the set in eq. \eqref{eq:gap}]{\label{fig2:c}
 \Drawdifference{2}{5}{7}{0.23} 
}
   \label{fig2}
\caption{Gap between the  necessary conditions in Theorem \ref{th:main} and the  sufficient conditions in Theorem \ref{th:2nonsquare} for $I_1<I_2<I_3$ (drawn for  $I_1=2$, $I_2=5$, $I_3=7$ and $\|\mathcal T\|=1$). The point $S_2$ has coordinates 
$(\frac{1}{I_1},\frac{1}{I_1},\frac{1}{I_1})$. The set in  plot (c) is the difference of the set in  plot (a) and the set in plot (b).}
\end{figure}

\begin{corollary}\label{th:2}
Let  $\sigma_1$, $\sigma_2$ and $\sigma_3$ satisfy  \eqref{eqeqs}--\eqref{eqeqstrivial} for $I_1=I_2=I_3=I\geq 2$. Then there exists an 
$I\times I\times I$ tensor $\mathcal T$ such that
\begin{enumerate}
\item all entries of $\mathcal T$ are non-negative;
\item $\mathcal T$ is all-orthogonal;
\item the largest ML singular values of  $\mathcal T$ are equal to $\sigma_1$, $\sigma_2$ and $\sigma_3$. 
 \end{enumerate}
\end{corollary}
Thus, the conditions in Theorem \ref{th:main} are not only necessary but also sufficient for $\sigma_1$, $\sigma_2$, and $\sigma_3$ to be feasible  largest ML singular values of a cubic third-order tensor. Figure \ref{fig:d} shows the set of feasible triplets $(\sigma_1^2,\sigma_2^2,\sigma_3^2)$ of an $I\times I\times I$ tensor.

We do not have a complete view on the  feasibility of points in \eqref{eq:gap}. 
In Section \ref{SX1Y1Z1} we obtain particular results on the  (non)feasibility of the  points $S(\frac{1}{I_1},\frac{1}{I_2}, \frac{1}{I_3})$, 
$X_1(1-\frac{1}{I_2}+\frac{1}{I_3},\frac{1}{I_2},\frac{1}{I_3})$, and
 $Y_1(\frac{1}{I_1},1-\frac{1}{I_1}+\frac{1}{I_3},\frac{1}{I_3})$. 
Namely, we show that if $I_1<I_2$ and $I_3=I_1I_2-1$, then the
 point $S$ is not feasible and  if $I_3=I_1I_2$, then  the point $S$ is feasible but the points 
 $X_1$ and $Y_1$ not. 
 
It worth mentioning a link with  scaled all-orthonormal tensors introduced recently in \cite{increadible_HOLQ}. 
Tensor $\mathcal T\in\mathbb C^{I_1\times\dots\times I_N}$ is {\em scaled all-orthonormal} \cite[Definition 2]{increadible_HOLQ} if at least $N-1$  of the $N$ matrices $\unf{T}{1}\unf{T}{1}^H,\dots,\unf{T}{N}\unf{T}{N}^H$ are multiples of the identity matrix. It is clear that
if the largest mode-$n$ singular value of a norm-$1$ tensor is $\frac{1}{\sqrt{I_n}}$, then all mode-$n$ singular  values are also  $\frac{1}{\sqrt{I_n}}$. Thus,  feasibility of a point
belonging to the segment $SX_1$ (resp. $SY_1$ or $SZ_1$)
is equivalent to the existence of a norm-$1$ $I_1\times I_2\times I_3$ tensor $\mathcal T$ such that
\begin{equation*}
\begin{split}
&\unf{T}{2}\unf{T}{2}^H = \frac{1}{I_2}\mathbf I_{I_2},\ 
\unf{T}{3}\unf{T}{3}^H = \frac{1}{I_3}\mathbf I_{I_3} \\
\ (\text{resp. } \unf{T}{1}\unf{T}{1}^H = \frac{1}{I_1}\mathbf I_{I_1},\  
&\unf{T}{3}\unf{T}{3}^H = \frac{1}{I_3}\mathbf I_{I_3}   \text{ or } \unf{T}{1}\unf{T}{1}^H = \frac{1}{I_1}\mathbf I_{I_1},\  
\unf{T}{2}\unf{T}{2}^H = \frac{1}{I_2}\mathbf I_{I_2}),
\end{split}
\end{equation*}
i.e., to the existence of a scaled all-orthonormal tensor $\mathcal T$.

The following results  generalize  Theorem \ref{th:main} and Corollary  \ref{th:2} for $N$th-order tensors. 
\begin{theorem}\label{th:main2}
Let $\sigma_1,\dots,\sigma_N$ denote the  largest ML  singular values of an $I_1\times\dots \times I_N$
tensor $\mathcal T$.
 Then 
 \begin{gather}
 \sigma_1^2+\dots+\sigma_{n-1}^2+\sigma_{n+1}^2+\dots+\sigma_{N}^2\leq (N-2)\|\mathcal T\|^2 + \sigma_n^2,\quad n=1,\dots,N,\label{eqeqsHO}\\
\|\mathcal T\|\geq \sigma_1\geq \frac{1}{\sqrt{I_1}}\|\mathcal T\|,\dots,\|\mathcal T\|\geq  \sigma_N\geq \frac{1}{\sqrt{I_N}}\|\mathcal T\|.\label{eqeqsHOobv}
  \end{gather}
\end{theorem}
\begin{theorem}\label{th:2Nthorder}
Let  $\sigma_1, \dots, \sigma_N$ satisfy  \eqref{eqeqsHO}--\eqref{eqeqsHOobv} for $I_1=\dots=I_N=I\geq 2$. Then there exists an 
$I\times \dots\times I$ tensor $\mathcal T$ such that
\begin{enumerate}
\item all entries of $\mathcal T$ are non-negative;
\item $\mathcal T$ is all-orthogonal;
\item the largest ML singular values of  $\mathcal T$ are equal to $\sigma_1,\dots,\sigma_N$. 
 \end{enumerate}
\end{theorem}
Thus, the conditions in Theorem \ref{th:main2} are not only necessary but also sufficient for $\sigma_1, \dots,\sigma_N$ to be feasible  largest ML singular values of an $I\times\dots\times I$ tensor. This result was independently proved for real $2\times\dots\times 2$ tensors in \cite{AnnaGramdeterminants}.

Theorems \ref{th:main}, \ref{th:2nonsquare}, \ref{th:main2}, and \ref{th:2Nthorder} are proved in Section \ref{sec:proofs}.

It is  natural  to ask  what happens if some inequalities in
\eqref{eqeqs} are replaced by equalities. Obviously, the three equalities in \eqref{eqeqs} hold if and only if  $\sigma_1=\sigma_2=\sigma_3=\|\mathcal T\|$, implying that 
$\unf{T}{1}$, $\unf{T}{2}$, and $\unf{T}{3}$ are rank-$1$ matrices.
Hence all the remaining ML singular values of $\mathcal T$ are zero.
Similarly, the two  equalities 
$ \sigma_1^2+\sigma_2^2= \|\mathcal T\|^2 + \sigma_3^2$ and $\sigma_1^2+\sigma_3^2= \|\mathcal T\|^2 + \sigma_2^2$ are equivalent to $\sigma_1=\|\mathcal T\|$ and $\sigma_2=\sigma_3$, 
implying that $\operatorname{rank}(\unf{T}{1})=1$ and $\operatorname{rank}(\unf{T}{2})=
\operatorname{rank}(\unf{T}{3})=:L$, i.e., $\mathcal T$ is an {\em ML rank-$(1,L,L)$ tensor}, where
$L\leq\min(I_2,I_3)$. It is clear that in this case the remaining
nonzero mode-$2$ and mode-$3$ singular values  of $\mathcal T$ also coincide and may take any positive values
whose squares sum up to $\|\mathcal T\|^2-\sigma_2^2$.
In Section \ref{Discussion} we characterize the tensors $\mathcal T$ for which the single equality $\sigma_1^2+\sigma_2^2= \|\mathcal T\|^2 + \sigma_3^2$ holds.
 We show that 
$\mathcal T$ is necessarily  equal to a sum of  ML rank-$(L_1,1,L_1)$  and ML rank-$(1,L_2,L_2)$ tensors and give a complete description of all its ML singular values. The description relies on a problem posed by H. Weyl in 1912: given the eigenvalues of two $n\times n$  Hermitian matrices $\mathbf A$ and $\mathbf B$,
what are all the possible eigenvalues of $\mathbf A+\mathbf B$? 
The following answer was conjectured by  A. Horn in 1962 \cite{Horn1962} and has been proved  through the development of the theory of honeycombs in \cite{Klyachko1998,KnutsonTao1999} (see also  \cite{Bhatia2001,Honeycombs2001}). Let
$$
\lambda_i(\cdot)\ \text{ denote the }i\text{th largest eigenvalue of a Hermitian matrix.}
$$
If
\begin{equation}\label{eq:alphabetagamma}
\alpha_i=\lambda_i(\mathbf A),\qquad \beta_i=\lambda_i(\mathbf B),\qquad  \gamma_i=\lambda_i(\mathbf A+\mathbf B),
\end{equation}
then   $\alpha_i$, $\beta_i$, and $\gamma_i$  satisfy the trivial  equality
\begin{equation}\label{eq:traceiequality}
\gamma_1+\dots+\gamma_n=\alpha_1+\dots+\alpha_n+\beta_1+\dots+\beta_n
\end{equation}
and the  list of linear inequalities
\begin{align}\label{eq:listofinequalities}
\sum\limits_{k\in K}\gamma_k\leq \sum\limits_{i\in I}\alpha_i + \sum\limits_{j\in J}\beta_j,\qquad (I,J,K)\in T_r^n,\qquad 1\leq r\leq n-1,
\end{align}
where $I=\{i_1,\dots,i_r\}$, $J=\{j_1,\dots,j_r\}$, $K=\{k_1,\dots,k_r\}$ are subsets of $\{1,\dots,n\}$ and $T_r^n$ denotes a particular finite set of  triplets $(I,J,K)$. (The construction of $T_r^n$ is given in  Appendix \ref{sec:appendixA}.) The inverse statement also holds: if $\alpha_i$, $\beta_i$, and $\gamma_i$
satisfy \eqref{eq:traceiequality} and \eqref{eq:listofinequalities}, then there exist $n\times n$ Hermitian matrices $\mathbf A$, $\mathbf B$, and $\mathbf C$ such that \eqref{eq:alphabetagamma} holds.

 We have the following results.
 \begin{theorem}\label{th:1LL and L1Lnew}
 Let  $\sigma_1^2+\sigma_2^2=\|\mathcal T\|^2+\sigma_3^2$. Then 
$\mathcal T$ is a sum of ML rank-$(L_1,1,L_1)$  and ML rank-$(1,L_2,L_2)$ tensors, where 
$L_1\leq\min (I_1,I_3)$ and $L_2\leq\min (I_2-1,I_3)$.
 \end{theorem}
\begin{theorem}\label{th:1LL and L1L}
Let  $\sigma_1^2+\sigma_2^2=\|\mathcal T\|^2+\sigma_3^2$. Then the values  
\begin{equation*}
\begin{split}
\sigma_1&=\sigma_{11}\geq \sigma_{12}\geq\dots\geq\sigma_{1I_1}\geq 0,\\
\sigma_2&=\sigma_{21}\geq \sigma_{22}\geq\dots\geq\sigma_{2I_2}\geq 0,\\
\sigma_3&=\sigma_{31}\geq \sigma_{32}\geq\dots\geq\sigma_{3I_3}\geq 0,
\end{split}
\end{equation*}
are the mode-$1$, mode-$2$, and mode-$3$ singular values of an $I_1\times I_2\times I_3$ tensor $\mathcal T$, respectively,  if and only if
$$
\sigma_{11}^2+\dots+\sigma_{1I_1}^2=
\sigma_{21}^2+\dots+\sigma_{2I_2}^2=
\sigma_{31}^2+\dots+\sigma_{3I_3}^2=\|\mathcal T\|^2,
$$
\begin{equation*}
\begin{split}
\sigma_{1i}&=0 \ \text{ for }\ i>\min(I_1,I_3), \\
\sigma_{2i}&=0 \ \text{ for }\ i>\min(I_2,I_3),
\end{split}
\end{equation*}
and \eqref{eq:traceiequality} and \eqref{eq:listofinequalities} hold for
\begin{equation}\label{eq:seventeen}
\begin{split}
\alpha_i=\begin{cases}
\sigma_{1i+1}^2,& i\leq\min(I_1,I_3)\\
 0,& \text{otherwise}
 \end{cases},\ 
\beta_i=\begin{cases}
\sigma_{2i+1}^2,& i\leq\min(I_2,I_3)\\
 0,& \text{otherwise}
 \end{cases},\ 
\gamma_i=\sigma_{3i+1}^2,
\end{split}
\end{equation}
and $n=I_3-1$.
\end{theorem}
\begin{example}
If $n=2$, then $T_1^2=\{(i,j,k): k=i+j-1, 1\leq i,j,k\leq 2\}=\{(1,1,1), (1,2,2), (2,1,2)\}$ (see Appendix \ref{sec:appendixA}). By Horn's conjecture, the equality $\gamma_1+\gamma_2=\alpha_1+\alpha_2+\beta_1+\beta_2$  together with  the  inequalities (also known as the Weyl inequalities) 
\begin{equation}\label{eq:Weyl}
\gamma_1\leq \alpha_1+\beta_1,\qquad \gamma_2\leq \alpha_1+\beta_2,\qquad \gamma_2\leq \alpha_2+\beta_1,
\end{equation}
characterize the values $\alpha_1,\alpha_2,\beta_1,\beta_2,\gamma_1,\gamma_2$ that can be eigenvalues of $2\times 2$  Hermitian matrices $\mathbf A$, $\mathbf B$, and $\mathbf A+\mathbf B$. Let $\sigma_{11}^2+\sigma_{21}^2=\|\mathcal T\|^2+\sigma_{31}^2$. From Theorem \ref{th:1LL and L1L} and \eqref{eq:Weyl} it follows that the values
$\sigma_{11}\geq \sigma_{12}\geq\sigma_{13}\geq 0$,  $\sigma_{21}\geq \sigma_{22}\geq\sigma_{23}\geq 0$,
and $\sigma_{31}\geq \sigma_{32}\geq\sigma_{33}\geq 0$, are the   mode-$1$, mode-$2$, and mode-$3$ singular values, respectively, of a $3\times 3\times 3$ tensor $\mathcal T$  if and only if 
\begin{align*}
\sigma_{11}^2+\sigma_{12}^2+\sigma_{13}^2=
\sigma_{21}^2+\sigma_{22}^2+\sigma_{23}^2=
\sigma_{31}^2+\sigma_{32}^2+\sigma_{33}^2=\|\mathcal T\|^2,\\
\sigma_{32}^2\leq \sigma_{12}^2+\sigma_{22}^2,\qquad
\sigma_{33}^2\leq \sigma_{12}^2+\sigma_{23}^2,\qquad
\sigma_{33}^2\leq \sigma_{13}^2+\sigma_{22}^2.
\end{align*}
\end{example}
Horn's Conjecture has recently also been linked to singular values of matrix unfoldings in the Tensor Train format \cite{SebastianMLSVDTT}.
\section{Proofs of Theorems \ref{th:main}, \ref{th:2nonsquare}, \ref{th:main2}, and \ref{th:2Nthorder} }\label{sec:proofs}
The following lemma will be used in the proof of Theorem \ref{th:main}. 
\begin{lemma}\label{lemma:main}
Let $\newA=(\newA_{ij})_{i,j=1}^{\newK}\in\mathbb C^{{\newK} \newn\times {\newK} \newn}$ be a positive semidefinite matrix  consisting of the blocks 
 $\newA_{ij}\in\mathbb C^{\newn\times \newn}$.
Then 
\begin{equation}\label{eq:eq1}
 \lambda_{max}(\newA_{11}+\dots+\newA_{{\newK} {\newK}})+
  \lambda_{max}(\newA)\leq 
  \operatorname{tr}(\newA)+
  \lambda_{max}(\mathbf \Phi(\newA)).
\end{equation}
where $\mathbf \Phi(\newA)$ denotes the ${\newK}\times {\newK}$ matrix  with the  entries $\left(\mathbf \Phi(\newA)\right)_{ij}=\operatorname{tr}(\newA_{ij})$
and $\lambda_{max}(\cdot)$ denotes the  largest eigenvalue of a matrix.
\end{lemma}
\begin{proof}
To get an idea of the proof we refer the reader to the  mathoverflow page \cite{fedja} where the case $I_3=2$ was discussed. Here we present a formal proof for $I_3\geq 2$.
Let $\newA=\sum\limits_{r=1}^R\mathbf w_{r}\mathbf w_{r}^H$, where $\mathbf w_r$ are orthogonal and
$\mathbf w_r=[\mathbf w_{1r}^T \ \dots\ \mathbf w_{{\newK} r}^T]^T$ with $\mathbf w_{kr}\in\mathbb C^{{\newK}}$.
First, we rewrite \eqref{eq:eq1} in terms of $\mathbf w_{kr}$, $1\leq k\leq {\newK}$, $1\leq r\leq R$.
WLOG, we can assume that  $\|\mathbf w_1\|=\max\limits_{r}\|\mathbf w_r\|$. Hence,
\begin{equation}\label{eq:eq2}
\lambda_{max}(\newA)=\|\mathbf w_1\|^2=\sum\limits_{k=1}^{\newK}\|\mathbf w_{k1}\|^2.
\end{equation}
It is clear that
$$
\newA_{ij} = \sum\limits_{r=1}^R\mathbf w_{ir}\mathbf w_{jr}^H,\qquad 1\leq i,j\leq {\newK}.
$$
Hence
\begin{equation}\label{eq:eq3}
\lambda_{max}(\newA_{11}+\dots+\newA_{{\newK} {\newK}})=\max\limits_{\|\mathbf x\|=1}\sum\limits_{k=1}^{\newK} (\newA_{kk}\mathbf x,\mathbf x)=
\max\limits_{\|\mathbf x\|=1}\sum\limits_{k=1}^{\newK}\sum\limits_{r=1}^R|(\mathbf w_{kr},\mathbf x)|^2.
\end{equation}
Since $\newA=\sum\limits_{r=1}^R\mathbf w_{r}\mathbf w_{r}^H$, it follows that
\begin{equation}\label{eq:eq4}
\operatorname{tr}(\newA)=
\sum\limits_{r=1}^R\|\mathbf w_r\|^2.
\end{equation}
Since 
$$
\mathbf \Phi(\newA)_{ij}=\operatorname{tr}(\newA_{ij})=\operatorname{tr}\left(\sum\limits_{r=1}^R
\mathbf w_{ir}\mathbf w_{jr}^H\right)=
\sum\limits_{r=1}^R \mathbf w_{jr}^H \mathbf w_{ir}=
\sum\limits_{r=1}^R \mathbf w_{ir}^T \mathbf w_{jr}^*,
$$
it follows that
\begin{equation}\label{eq:eq4.5}
\begin{split}
\mathbf \Phi(\newA)=
&\sum\limits_{r=1}^R
\left[
\begin{matrix}
\mathbf w_{1r}^T\mathbf w_{1r}^*&\dots& \mathbf w_{1r}^T\mathbf w_{{\newK} r}^*\\
\vdots&\dots&\vdots\\
\mathbf w_{{\newK} r}^T\mathbf w_{1r}^*&\dots& \mathbf w_{{\newK} r}^T\mathbf w_{{\newK} r}^*
\end{matrix}
\right]=\\
&\sum\limits_{r=1}^R
\left[
\begin{matrix}
\mathbf w_{1r}^T\\
\vdots\\
\mathbf w_{{\newK} r}^T
\end{matrix}
\right]
\left[
\begin{matrix}
\mathbf w_{1r}^*&\dots&\mathbf w_{{\newK} r}^*
\end{matrix}
\right]=\sum\limits_{r=1}^R\mathbf W_r^T\mathbf W_r^*,`
\end{split}
\end{equation}
where
$$
\mathbf W_r:=[\mathbf w_{1r}\ \dots\ \mathbf w_{{\newK} r}]\in\mathbb C^{\newn\times {\newK}}.
$$
Now we prove \eqref{eq:eq1}.
By \eqref{eq:eq2}, \eqref{eq:eq3}, the Cauchy inequality, and \eqref{eq:eq4},
\begin{equation}\label{eq:eq5}
\begin{split}
&\lambda_{max}(\newA)+\lambda_{max}(\newA_{11}+\dots\newA_{{\newK} {\newK}})=\\
&\|\mathbf w_1\|^2 + \max\limits_{\|\mathbf x\|=1}\left[
\sum\limits_{k=1}^{\newK}|(\mathbf w_{k1},\mathbf x)|^2+
\sum\limits_{k=1}^{\newK}\sum\limits_{r=2}^R|(\mathbf w_{kr},\mathbf x)|^2\right]\leq\\
&\|\mathbf w_1\|^2 + \max\limits_{\|\mathbf x\|=1}
\left[
\sum\limits_{k=1}^{\newK}|(\mathbf w_{k1},\mathbf x)|^2
\right] 
+ \sum\limits_{r=2}^R\|\mathbf w_r\|^2=
\operatorname{tr}(\newA) + 
\max\limits_{\|\mathbf x\|=1}\left[
\sum\limits_{k=1}^{\newK}|(\mathbf w_{k1},\mathbf x)|^2\right].
\end{split}
\end{equation}
To complete the proof of \eqref{eq:eq1} we should show that 
$$
\max\limits_{\|\mathbf x\|=1}\left[
\sum\limits_{k=1}^{\newK}|(\mathbf w_{k1},\mathbf x)|^2\right]\leq\lambda_{\max}(\mathbf \Phi(\newA)).
$$
This can be done as follows
\begin{equation}\label{equation18}
\begin{split}
\max\limits_{\|\mathbf x\|=1}\left[
\sum\limits_{k=1}^{\newK}|(\mathbf w_{k1},\mathbf x)|^2\right]=
\max\limits_{\|\mathbf x\|=1}\left[
\sum\limits_{k=1}^{\newK}\mathbf x^H\mathbf w_{k1}\mathbf w_{k1}^H\mathbf x\right]=
\lambda_{max}\left(
\mathbf W_1\mathbf W_1^H\right)=
\\
\lambda_{max}\left(
\mathbf W_1^H\mathbf W_1\right)\leq \lambda_{max}\left(\sum\limits_{r=1}^R\mathbf W_r^H\mathbf W_r\right)=
\lambda_{max}(\mathbf\Phi(\newA)^*)=\lambda_{max}(\mathbf\Phi(\newA)).
\end{split}
\end{equation}
\end{proof}
Now we are ready to prove Theorem \ref{th:main}.
\begin{proof}[Proof of Theorem \ref{th:main}]
The  three inequalities in \eqref{eqeqstrivial} are obvious. We prove that $\sigma_1^2+\sigma_2^2\leq \|\mathcal T\|^2 + \sigma_3^2$. The
proofs of the  inequalities $\sigma_1^2+\sigma_3^2\leq \|\mathcal T\|^2 + \sigma_2^2$ and $\sigma_2^2+\sigma_3^2\leq \|\mathcal T\|^2 + \sigma_1^2$ can be obtained in a similar way.

By definition of ML singular values,
 \begin{align*}
 \sigma_1^2&=\lambda_{max}(\unf{T}{1}\unf{T}{1}^H)=\lambda_{max}(\mathbf T_1\mathbf T_1^H+\dots+\mathbf T_{I_3}\mathbf T_{I_3}^H),\\
 \sigma_2^2&=\lambda_{max}(\unf{T}{2}^H\unf{T}{2})=
 \lambda_{max}(\unf{T}{2}^T\unf{T}{2}^*) = 
 \lambda_{max}(\newA),
 \end{align*}
 where
 $$
  \newA= \unf{T}{2}^T\unf{T}{2}^*=
   \left[
    \begin{matrix}
    \mathbf T_1\mathbf T_1^H&\dots& \mathbf T_1\mathbf T_{I_3}^H\\
    \vdots&\dots&\vdots\\
    \mathbf T_{I_3}\mathbf T_1^H&\dots& \mathbf T_{I_3}\mathbf T_{I_3}^H
    \end{matrix}
    \right].
    $$
Since 
 $\operatorname{vec}(\mathbf T_i)^T(\operatorname{vec}(\mathbf T_j)^T)^H=
 \operatorname{tr}(\mathbf T_i\mathbf T_j^H)$, it follows that
  \begin{align*}
   \sigma_3^2&=\lambda_{max}(\unf{T}{3}\unf{T}{3}^H)=   \lambda_{max}(\mathbf \Phi(\newA)),
 \end{align*}
 where
 $$
 \mathbf \Phi(\newA) = 
 \left[
 \begin{matrix}
 \operatorname{tr}(\mathbf T_1\mathbf T_1^H)&\dots& \operatorname{tr}(\mathbf T_1\mathbf T_{I_3}^H)\\
 \vdots&\dots&\vdots\\
 \operatorname{tr}(\mathbf T_{I_3}\mathbf T_1^H)&\dots& \operatorname{tr}(\mathbf T_{I_3}\mathbf T_{I_3}^H)
 \end{matrix}
 \right].
 $$
 Since $\|\mathcal T\|^2=\operatorname{tr}(\newA)$, the inequality $\sigma_1^2+\sigma_2^2\leq \|\mathcal T\|^2 + \sigma_3^2$ is equivalent to
 $$
 \lambda_{max}(\mathbf T_1\mathbf T_1^H+\dots+\mathbf T_{I_3}\mathbf T_{I_3}^H)+
  \lambda_{max}(\newA)\leq 
  \operatorname{tr}(\newA)+
  \lambda_{max}(\mathbf \Phi(\newA)),
 $$
 which holds by Lemma \ref{lemma:main}. 
\end{proof}

\begin{proof}[Proof of Theorem \ref{th:2nonsquare}]
The proof consists of three steps. In the first step we construct all-orthogonal and non-negative $I_1\times I_2\times I_3$ tensors $\mathcal S_2$, $\mathcal X_2$, $\mathcal Y_2$, $\mathcal Z_2$, and $\mathcal N$ whose squared largest ML singular values  are the coordinates of
$S_2(\frac{1}{I_1},\frac{1}{I_1},\frac{1}{I_1})$, $X_2(1,\frac{1}{I_2},\frac{1}{I_2})$, $Y_2(\frac{1}{I_1},1,\frac{1}{I_1})$, $Z_2(\frac{1}{I_1},\frac{1}{I_1},1)$, and $N(1,1,1)$, respectively  (see Figure \ref{fig2:b}). Then  we show that  because of the zero patterns of $\mathcal S_2$, $\mathcal X_2$, $\mathcal Y_2$, $\mathcal Z_2$, and $\mathcal N$, 
the tensor
\begin{equation}\label{eq:212}
\mathcal T=\left(t_{S_2}\mathcal S_2^2 + t_{X_2}\mathcal X_2^2+t_{Y_2}\mathcal Y_2^2+t_{Z_2}\mathcal Z_2^2+t_{N}\mathcal N^2\right)^{\frac{1}{2}},
\end{equation}
is  all-orthogonal for any non-negative values $t_{S_2}$, $t_{X_2}$, $t_{Y_2}$, $t_{Z_2}$, $t_{N}$.
The superscripts ``$2$'' and ``$\frac12$'' in \eqref{eq:212}  denote the  entrywise operations.
Finally, in the third step, we find non-negative values $t_{S_2}$, $t_{X_2}$, $t_{Y_2}$, $t_{Z_2}$, $t_N$ such that $\mathcal T$ is norm-$1$ tensor whose  squared largest ML singular values  are equal to $\sigma_1^2$, $\sigma_2^2$, and $\sigma_3^2$.

{\em Step 1.} Let $\pi$ denote the cyclic permutation
$
\pi:1\rightarrow I_1\rightarrow I_1-1\rightarrow\dots\rightarrow 2\rightarrow 1
$.
The tensors  $\mathcal S_2$, $\mathcal X_2$, $\mathcal Y_2$, and $\mathcal Z_2$ are defined by 
\begin{align*}
\mathcal S_{2,ijk}&=
\begin{cases}
\frac{1}{I_1},& \text{if } j=\pi^{k-1}(i)\text{ and } 1\leq i,k\leq I_1,\\
0,& \text{otherwise},
\end{cases}\\
\mathcal X_{2,ijk}&=
\begin{cases}
\frac{1}{\sqrt{I_2}},& \text{if }  j=\pi^{k-1}(i),\ i=1,  \text{ and } 1\leq k\leq I_1,\\
\frac{1}{\sqrt{I_2}},& \text{if } i=1 \text{ and } I_1< j=k\leq I_2,\\
0,& \text{otherwise},
\end{cases}\\
\mathcal Y_{2,ijk}&=
\begin{cases}
\frac{1}{\sqrt{I_1}},& \text{if } j=\pi^{k-1}(i),\  j=1, \text{ and } 1\leq k\leq I_1,\\
0,& \text{otherwise},
\end{cases}\\
\mathcal Z_{2,ijk}&=
\begin{cases}
\frac{1}{\sqrt{I_1}},& \text{if }  j=\pi^{k-1}(i),\ k=1,\text{ and } 1\leq i\leq I_1,\\
0,& \text{otherwise},
\end{cases}
\end{align*}
and the tensor $\mathcal N$, by definition, has only one nonzero entry, $\mathcal N_{111}=1$.
 For instance, if $I_1=I_2=I_3=2$, then the first matrix unfoldings of 
$\mathcal S_2$, $\mathcal X_2$, $\mathcal Y_2$,  $\mathcal Z_2$, and $\mathcal N$ have the form
\begin{equation*}
\begin{split}
\mathbf S_{2,(1)}&=\frac12\left[\begin{matrix}1&0&0&1\\0&1&1&0\end{matrix}\right],\
\mathbf X_{2,(1)}=\frac{1}{\sqrt{2}}\left[\begin{matrix}1&0&0&0\\0&1&0&0\end{matrix}\right],\\
\mathbf Y_{2,(1)}&=\frac{1}{\sqrt{2}}\left[\begin{matrix}1&0&0&0\\0&0&1&0\end{matrix}\right],\
\mathbf Z_{2,(1)}=\frac{1}{\sqrt{2}}\left[\begin{matrix}1&0&0&1\\0&0&0&0\end{matrix}\right],\
\mathbf N_{(1)}=\left[\begin{matrix}1&0&0&0\\0&0&0&0\end{matrix}\right].
\end{split} 
\end{equation*}

{\em Step 2.} 
It is clear that the $(i,j,k)$th entry of a linear combination of $\mathcal S_2^2$, $\mathcal X_2^2$, $\mathcal Y_2^2$, $\mathcal Z_2^2$, and $\mathcal N^2$ may be nonzero only if 
$$
j=\pi^{k-1}(i)\text{ and } 1\leq i,k\leq I_1\qquad\qquad\text{or}\qquad\qquad
 i=1 \text{ and } I_1< j=k\leq I_2.
$$
The same is also true for  $\mathcal T$ defined  in \eqref{eq:212}. One can easily check that
each column of $\unf{T}{1}$, $\unf{T}{2}$, and $\unf{T}{3}$  contains at most one nonzero entry, implying that $\mathcal T$ is all-orthogonal tensor.

{\em Step 3.} From the construction of the all-orthogonal tensors $\mathcal S_2$, $\mathcal X_2$, $\mathcal Y_2$,
$\mathcal Z_2$, and $\mathcal N$ it follows that their largest ML singular values are equal to the Frobenius norms of  the first  rows of their  matrix unfoldings. 
Thus,   the same property should also hold for $\mathcal T$ whenever the values $t_{S_2}$, $t_{X_2}$, $t_{Y_2}$, $t_{Z_2}$, and $t_N$ are non-negative.
Now the result follows from the fact that the  polyhedron in Figure \ref{fig2:b}
is the convex hull of the points $S_2$, $X_2$, $Y_2$, $Z_2$, and $N$.
We can also write the values of $t_{S_2}$, $t_{X_2}$, $t_{Y_2}$, $t_{Z_2}$, and $t_N$ explicitly. 
We set
\begin{equation*}
f(\sigma_1^2,\sigma_2^2,\sigma_3^2):=(I_1I_2+I_2-2I_1)\sigma_1^2+
(I_1-1)I_2\sigma_2^2+
(I_1-1)I_2\sigma_3^2+(2-I_1I_2-I_2).
\end{equation*}
If $(\sigma_1^2,\sigma_2^2, \sigma_3^2)$ belongs to the tetrahedron $X_2Y_2Z_2N$, i.e.,
$f(\sigma_1^2,\sigma_2^2,\sigma_3^2)\geq 0$, then
\begin{align*}
t_{X_2} &= \frac{I_2}{2(I_2-1)}(1+\sigma_1^2-\sigma_2^2-\sigma_3^2), \  t_{Y_2} = \frac{I_1}{2(I_1-1)}(1+\sigma_2^2-\sigma_1^2-\sigma_3^2),\\
t_{Z_2} &= \frac{I_1}{2(I_1-1)}(1+\sigma_3^2-\sigma_1^2-\sigma_2^2), \\
t_N &= 1-t_{X_2}-t_{Y_2}-t_{Z_2}= \frac{f(\sigma_1^2,\sigma_2^2,\sigma_3^2)}{2(I_1-1)(I_2-1)},\qquad t_{S_2}=0.
\end{align*}
If $(\sigma_1^2,\sigma_2^2, \sigma_3^2)$ belongs to the tetrahedron $X_2Y_2Z_2S_2$, i.e., 
$f(\sigma_1^2,\sigma_2^2,\sigma_3^2)\leq 0$,
then
\begin{align*}
t_{X_2} &= \frac{I_1}{I_1-1}(\sigma_1^2-\frac{1}{I_1}),\\
t_{Y_2} &= \frac{I_1}{I_1-1}(\sigma_2^2-\frac{1}{I_1})+\frac{(I_2-I_1)I_1}{(I_1^2-1)I_2}(\sigma_1^2-\frac{1}{I_1}),\\
t_{Z_2} &= \frac{I_1}{I_1-1}(\sigma_3^2-\frac{1}{I_1})+\frac{(I_2-I_1)I_1}{(I_1^2-1)I_2}(\sigma_1^2-\frac{1}{I_1}),\\
t_{S_2} &= 1-t_{X_2}-t_{Y_2}-t_{Z_2}=\frac{-f(\sigma_1^2,\sigma_2^2,\sigma_3^2)I_1}{I_2(I_1-1)^2},\qquad t_N=0.
\end{align*}
\end{proof}
\begin{proof}[Proof of Theorem \ref{th:main2}]
The inequalities in \eqref{eqeqsHOobv} are obvious. We prove that 
\begin{equation}\label{eq:long_Nthorderinequality}
\sigma_1^2+\dots+\sigma_{N-1}^2\leq (N-2)\|\mathcal T\|^2 + \sigma_N^2.
\end{equation}
The proofs of the remaining $N-1$  inequalities in \eqref{eqeqsHO} can be obtained in a similar way. 

The proof of \eqref{eq:long_Nthorderinequality} consists of two steps. In the first step we reshape $\mathcal T$ into  third-order tensors $\mathcal T^{[1]},\dots,\mathcal T^{[N-2]}$ and  compute their matrix unfoldings.
In this step we will make use of  \eqref{eq:mode_n-unfolding} for $N=3$. For the  reader's convenience and for a future reference here we write a third-order version of \eqref{eq:mode_n-unfolding}  explicitly:  if $\mathcal X\in\mathbb C^{I\times J\times K}$, then for all values of indices $i$, $j$, and $k$
\begin{equation}\label{eq:3rdorderanalogue}
\begin{split}
\text{the }(i,j+(k-1)J)\text{th entry of } \unf{X}{1} &=
\text{the }(j,i+(k-1)I)\text{th entry of } \unf{X}{2} =\\
\text{the }(k,i+(j-1)I)\text{th entry of } \unf{X}{3} &=
\text{the }(i,j,k)\text{th entry of }\mathcal X.
\end{split}
\end{equation}
In the second step, we apply the first inequality in \eqref{eqeqs} to each tensor $\mathcal T^{[n]}$, then we
sum up the obtained inequalities and show that the result coincides with inequality \eqref{eq:long_Nthorderinequality}.

{\em Step 1.} Let $n\in\{1,\dots,N-2\}$. A third-order tensor $\mathcal T^{[n]}\in\mathbb C^{I_1\cdots I_{n}\times I_{n+1}\times I_{n+2}\cdots I_N}$ is constructed as follows: 
\begin{multline*} 
\text{the } (i_1+\sum_{k=2}^n(i_k-1)\prod_{l=1}^{k-1}I_l, i_{n+1}, i_{n+2}+\sum_{k=n+3}^N(i_k-1)\prod_{l=n+2}^{k-1}I_l)\text{th entry of }\mathcal T^{[n]}\\
 \text{ is equal to the }(i_1,\dots, i_N)\text{th entry of }\mathcal T. 
\end{multline*}
Now we apply \eqref{eq:3rdorderanalogue} for $\mathcal X=\mathcal T^{[n]}$ and 
$$
i= i_1+\sum_{k=2}^n(i_k-1)\prod_{l=1}^{k-1}I_l,\quad j= i_{n+1},\quad k=i_{n+2}+\sum_{k=n+3}^N(i_k-1)\prod_{l=n+2}^{k-1}I_l.
$$ 
After simple algebraic manipulations, we obtain that
\begin{equation}\label{eq:16}
\begin{split}
&\text{the } (i_1+\sum_{k=2}^n(i_k-1)\prod_{l=1}^{k-1}I_l, i_{n+1}+\sum_{k=n+2}^N(i_k-1)\prod_{l=n+1}^{k-1}I_l)\text{th entry of }\unf{T}{1}^{[n]}=\\
&\text{the } (i_{n+1}, 1+\sum_{\substack{k=2\\ k\ne n+1}}^N (i_k-1)\prod_{\substack{l=1\\ l\ne n+1}}^{k-1}I_l)\text{th entry of }\unf{T}{2}^{[n]}=\\
&\text{the } (i_{n+2}+\sum_{k=n+3}^N(i_k-1)\prod_{l=n+2}^{k-1}I_l, i_1+\sum_{k=2}^{n+1}(i_k-1)\prod_{l=1}^{k-1}I_l)\text{th entry of }\unf{T}{3}^{[n]}=\\
&\text{the }(i_1,\dots, i_N)\text{th entry of }\mathcal T.
\end{split}
\end{equation}
\qquad
{\em Step 2.} From  \eqref{eq:16} and \eqref{eq:mode_n-unfolding} it follows that
\begin{align}
\unf{T}{1}^{[1]}&=\unf{T}{1},\label{eq:first}\\
\unf{T}{2}^{[n]}&=\unf{T}{n+1},\qquad 1\leq n\leq N-2,\label{eq:second}\\
\unf{T}{3}^{[N-2]}&=\unf{T}{N}.\label{eq:third}
\end{align}
Comparing the expressions of $\unf{T}{1}^{[n]}$ and $\unf{T}{3}^{[n]}$ in \eqref{eq:16}, we obtain that
\begin{equation}\label{eq:last}
\unf{T}{3}^{[n]}=\left(\unf{T}{1}^{[n+1]}\right)^T,\qquad 1\leq n\leq N-3.
\end{equation}
By Theorem \ref{th:main}, for every $n\in\{1,\dots,N-2\}$
\begin{equation}\label{eq:byTh1}
\sigma^2_{max}(\unf{T}{1}^{[n]})+\sigma^2_{max}(\unf{T}{2}^{[n]})\leq \|\mathcal T^{[n]}\|^2+ \sigma^2_{max}(\unf{T}{3}^{[n]})=
\|\mathcal T\|^2+\sigma^2_{max}(\unf{T}{3}^{[n]}),
\end{equation}
where $\sigma_{max}(\cdot)$ denotes the largest singular value of a matrix. Substituting \eqref{eq:first}--\eqref{eq:last} into
\eqref{eq:byTh1} we obtain
\begin{align*}
\sigma_1^2+\sigma_2^2 &\leq \|\mathcal T\|^2+\sigma^2_{max}(\unf{T}{3}^{[1]}) = \|\mathcal T\|^2+\sigma^2_{max}(\unf{T}{1}^{[2]}),\ n=1,\\
\sigma^2_{max}(\unf{T}{1}^{[2]})+\sigma_3^2 &\leq \|\mathcal T\|^2+\sigma^2_{max}(\unf{T}{3}^{[2]}) = \|\mathcal T\|^2+\sigma^2_{max}(\unf{T}{1}^{[3]}),\ n=2,\\
&\vdots\\
\sigma^2_{max}(\unf{T}{1}^{[N-3]})+\sigma_{N-2}^2 &\leq \|\mathcal T\|^2+\sigma^2_{max}(\unf{T}{3}^{[N-3]}) = \|\mathcal T\|^2+\sigma^2_{max}(\unf{T}{1}^{[N-2]}),\  n=N-3,\\
\sigma^2_{max}(\unf{T}{1}^{[N-2]})+\sigma_{N-1}^2 &\leq \|\mathcal T\|^2+\sigma^2_{max}(\unf{T}{3}^{[N-2]}) = \|\mathcal T\|^2+\sigma^2_{N},\ n=N-2.
\end{align*}
Summing up the above inequalities and canceling identical terms on the left- and right-hand side we obtain \eqref{eq:long_Nthorderinequality}. 
\end{proof}
\begin{proof}[Proof of Theorem \ref{th:2Nthorder}]
It can be checked that a polyhedron described by the inequalities in \eqref{eqeqsHO}--\eqref{eqeqsHOobv}  is a convex hull of $2^N-N$ points 
\begin{equation}
V=\{
(\alpha_1,\dots,\alpha_N), \quad \alpha_n\in\left\{\frac{1}{I}, 1\right\}\ \text{and at least two of } \alpha\text{-s are equal to }\frac{1}{I}\}.\label{alphas}
\end{equation}
To show that each point of the polyhedron is feasible we proceed as in the proof of Theorem \ref{th:2nonsquare}.

First, for each  $(\alpha_1,\dots,\alpha_N)\in V$  we   construct an all-orthogonal and non-negative $I\times \dots\times I$ tensor $P^{\alpha_1,\dots,\alpha_N}$ whose squared largest ML singular values  are $\alpha_1,\dots,\alpha_N$.

Let $\pi$ denote the cyclic permutation
$
\pi:1\rightarrow I\rightarrow I-1\rightarrow\dots\rightarrow 2\rightarrow 1
$.
The tensor  $\mathcal P^{\frac{1}{I},\dots, \frac{1}{I}}$ is defined by 
\begin{equation*}
\mathcal P^{\frac{1}{I},\dots, \frac{1}{I}}_{i_1,\dots,i_N}=
\begin{cases}
I^{-\frac{N-1}{2}},& \text{if } i_2=\pi^{i_3+\dots+i_N-N+2}(i_1),\\
0,& \text{otherwise},
\end{cases}
\end{equation*}
and the tensor $\mathcal P^{1,\dots, 1}$, by definition, has only one nonzero entry, $\mathcal P^{1\dots 1}_{1,\dots,1}=1$. Let
$(\alpha_1,\dots,\alpha_N)\in V\setminus\{(\frac1I,\dots,\frac1I), (1,\dots,1)\}$ and $j_1,\dots,j_k$ denote all indices such that 
$\alpha_{j_1}=\dots=\alpha_{j_k}=1$. Then the tensor $P^{\alpha_1,\dots,\alpha_N}$ is defined by
\begin{equation*}
\mathcal P^{\alpha_1,\dots,\alpha_N}_{i_1,\dots,i_N}=
\begin{cases}
I^{-\frac{N-1-k}{2}},& \text{if } i_2=\pi^{i_3+\dots+i_N-N+2}(i_1) \text{ and } i_{j_1}=\dots=i_{j_k}=1,\\
0,& \text{otherwise}.
\end{cases}
\end{equation*}
For instance, if $N=4$ and $I=2$, then the first matrix unfolding of $\mathcal P^{\frac{1}{I},\dots, \frac{1}{I}}$
is given by
$$
\mathcal P^{\frac{1}{I},\dots, \frac{1}{I}}_{(1)} = \frac{1}{2\sqrt{2}}\left[
\begin{array}{*{27}c}
1&0&0&1&0&1&1&0\\
0&1&1&0&1&0&0&1\\
\end{array}\right]
$$
and the  first matrix unfoldings of the remaining tensors $\mathcal P^{\alpha_1,\dots,\alpha_N}$
can be obtained from $\mathcal P^{\frac{1}{I},\dots, \frac{1}{I}}_{(1)}$ by rescaling and introducing more zeros.

It is clear that the $(i_1,\dots,i_N)$th entry of a linear combination of $\mathcal P^{\frac{1}{I},\dots,\frac{1}{I}},\dots,\mathcal P^{1,\dots,1}$ may be nonzero only if 
$$
i_2=\pi^{i_3+\dots+i_N-N+2}(i_1) .
$$
The same is also true for  $\mathcal T$ defined  by
\begin{equation*}
\mathcal T=\left(\sum\limits_{(\alpha_1,\dots,\alpha_N)\in V}   t_{\alpha_1,\dots,\alpha_N} P^{{\alpha_1,\dots,\alpha_N}^2} \right)^{\frac{1}{2}},
\end{equation*}
where, as before, the superscripts ``$2$'' and ``$\frac12$''    denote the  entrywise operations.
One can easily check that each column of $\unf{T}{1},\dots,\unf{T}{N}$  contains at most one nonzero entry, implying that $\mathcal T$ is all-orthogonal tensor.
Finally, from the construction of the all-orthogonal tensors $P^{\alpha_1,\dots,\alpha_N}$ it follows that their largest ML singular values are equal to the Frobenius norms of  the first  rows of their  matrix
 unfoldings.
Thus,   the same property should also hold for $\mathcal T$ whenever the values $t_{\alpha_1,\dots,\alpha_N}$ are non-negative. Now the result follows from the fact that the  polyhedron 
described by the inequalities in \eqref{eqeqsHO}--\eqref{eqeqsHOobv}  is a convex hull of
points in $V$.
\end{proof}
Note that in the proof of  Theorem \ref{th:2Nthorder} the constructed tensor ${\mathcal T}$ has squared singular values in the $n$th mode equal to $\sigma_n^2, \frac{1}{I-1}(1-\sigma_n^2),\dots,\frac{1}{I-1}(1-\sigma_n^2)$, i.e., the $I-1$ smallest singular values in the $n$th mode are equal.
\section{Results on feasibility and non-feasibility of the points \texorpdfstring{$S$}{S}, \texorpdfstring{$X_1$}{X1}, and \texorpdfstring{$Y_1$}{Y1}}\label{SX1Y1Z1}
Throughout this subsection  we assume that $\mathcal T$ is a norm-$1$ tensor.

In the following example we show that it may happen that $S$ is the only  feasible point in the plane through  the points $S$, $X_1$, and $Y_1$, i.e., the plane $\sigma_3^2=\frac{1}{I_3}$. 
\begin{example}\label{ex:6}
Let $I_3=I_1I_2$ and $\mathcal T\in\mathbb C^{I_1\times I_2\times I_3}$. 
Assume that $\sigma_3^2=\frac{1}{I_3}$. Then $\unf{T}{3}^H\unf{T}{3}=\frac{1}{I_3}\mathbf I_{I_3}$. Since $\unf{T}{3}$ is a square matrix, it follows that $\unf{T}{3}$ is a scalar multiple of a unitary matrix, $\unf{T}{3}=\frac{1}{\sqrt{I_3}}\mathbf U$. One can easily verify (see \cite[p. 65]{increadible_HOLQ}), that
$\unf{T}{1}^H\unf{T}{1}=\frac{1}{I_1}\mathbf I_{I_1}$ and 
$\unf{T}{2}^H\unf{T}{2}=\frac{1}{I_2}\mathbf I_{I_2}$. Hence, $\sigma_1^2=\frac{1}{I_1}$ and $\sigma_2^2=\frac{1}{I_2}$. 
Thus, the points $X_1$ and $Y_1$ are not feasible.
\end{example}
From Example \ref{ex:6} it follows that the point $S$ is feasible if $I_1=2$, $I_2=3$, and $I_3=6$.
The point $S$ is also feasible if $I_1=2$, $I_2=3$, and $I_3=4$. Indeed, let 
$\mathcal T$ be an $2\times 3\times 4$ tensor with  mode-$3$ matrix
unfolding
\begin{align*}
\unf{T}{3} = 
\frac{1}{2\sqrt{3}}
\begin{bmatrix}
1+\sqrt{3} & 0          & 0          & 1-\sqrt{3} & -2 & 0\\
0          & 1+\sqrt{3} & 1-\sqrt{3} & 0          &  0 & 2\\
0          & 1-\sqrt{3} & 1+\sqrt{3} & 0          &  0 & 2\\
1-\sqrt{3} & 0          & 0          & 1+\sqrt{3} & -2 & 0
\end{bmatrix}.
\end{align*}
Then one can also easily verify that $\unf{T}{1}\unf{T}{1}^H=\frac12 \mathbf I_2$, $\unf{T}{2}\unf{T}{2}^H=\frac13 \mathbf I_3$, and
$\unf{T}{3}\unf{T}{3}^H=\frac14 \mathbf I_4$.
The following result implies that in the ``intermediate'' case  $I_1=2$, $I_2=3$, and $I_3=5$ the point  $S$ is not feasible.
\begin{theorem}\label{th:235}
Let $I_3=I_1I_2-1$, $\mathcal T\in\mathbb C^{I_1\times I_2\times I_3}$, and  $\unf{T}{3}\unf{T}{3}^H=\frac{1}{I_3}\mathbf I_{I_3}$.
Then the following statements hold:
\begin{enumerate}[label=\textnormal{(\roman*)}] 
\item if  $\unf{T}{1}\unf{T}{1}^H=\frac{1}{I_1}\mathbf I_{I_1}$, then  $I_1\leq I_2$;\label{itm2:1}
\item if  $\unf{T}{2}\unf{T}{2}^H=\frac{1}{I_2}\mathbf I_{I_2}$, then  $I_2\leq I_1$;\label{itm2:2}
\item if the point $S$ is feasible, then $I_1=I_2$.\label{itm2:3}
\end{enumerate}
\end{theorem}
\begin{proof}
\ref{itm2:1}\ Let $\unf{T}{3}=[\mathbf t_1\ \dots\ \mathbf t_{I_1I_2}]$. Then the identity $\unf{T}{1}\unf{T}{1}^H=\frac{1}{I_1} \mathbf I_{I_1}$ is equivalent to the system
\begin{equation}
\begin{split}
&\mathbf t_{i_1}^H\mathbf t_{i_2}+\mathbf t_{I_1+i_1}^H\mathbf t_{I_1+i_2}+\dots+\mathbf t_{I_1(I_2-1)+i_1}^H\mathbf t_{I_1(I_2-1)+i_2} = 0, \\
&\|\mathbf t_{i_1}\|^2+\|\mathbf t_{I_1+i_1}\|^2 +\dots+ \|\mathbf t_{I_1(I_2-1)+i_1}\|^2=\frac{1}{I_1},\qquad\qquad
1\leq i_1<i_2\leq I_1.   
\end{split}\label{eq:contr1}
\end{equation}
Since $\unf{T}{3}\unf{T}{3}^H=\frac{1}{I_3}\mathbf I_{I_3}$, the matrix $\sqrt{I_3}\unf{T}{3}\in\mathbb C^{I_3\times I_1I_2}$ can be extended to a unitary matrix 
$\sqrt{I_3}\begin{bmatrix}
\unf{T}{3}\\
\mathbf a^T
\end{bmatrix}\in\mathbb C^{I_1I_2\times I_1I_2}
$, where  $\mathbf a\in\mathbb C^{I_1I_2}$ is a vector such that $\unf{T}{3}\mathbf a^*=\mathbf 0$ and $\|\mathbf a\|^2=\frac{1}{I_3}$.
Hence,
$$
\begin{bmatrix}
\unf{T}{3}^H&\mathbf a^*
\end{bmatrix}
\begin{bmatrix}
\unf{T}{3}\\
\mathbf a^T
\end{bmatrix}
=\frac{1}{I_3}\mathbf I_{I_2I_3} 
$$
or
\begin{equation}\label{eq:contr2}
\mathbf t_i^H\mathbf t_j+\bar{a}_ia_j=0 \ \text{ for } i\ne j\ \text{ and }\
\|\mathbf t_i\|^2+|a_i|^2=\frac{1}{I_3},\qquad 1\leq i<j\leq I_1I_2.
\end{equation}
From \eqref{eq:contr1}--\eqref{eq:contr2} it follows that
\begin{equation*}
\begin{split}
& \bar{a}_{i_1} a_{i_2}+\bar{a}_{I_1+i_1} a_{I_1+i_2}+\dots+\bar{a}_{I_1(I_2-1)+i_1} a_{I_1(I_2-1)+i_2} = 0, \\
&|a_{i_1}|^2+|a_{I_1+i_1}|^2 +\dots+ |a_{I_1(I_2-1)+i_1}|^2=\frac{1}{I_1},\qquad\qquad
1\leq i_1<i_2\leq I_1.   
\end{split}
\end{equation*}
Thus, the  vectors 
$$
[a_{i}\ a_{I_1+i}\ \dots\ a_{I_1(I_2-1)+i}]^T\in\mathbb C^{I_2},\qquad\qquad 1\leq i\leq I_1
$$ 
are nonzero and mutually  orthogonal. Hence, $I_1\leq I_2$.

\ref{itm2:2}\ The proof is similar to the proof of \ref{itm2:1}.

\ref{itm2:3}\ Since $S$ is feasible, it follows that $\unf{T}{1}\unf{T}{1}^H=\frac{1}{I_1}\mathbf I_{I_1}$ and
$\unf{T}{2}\unf{T}{2}^H=\frac{1}{I_2}\mathbf I_{I_2}$. Hence, by \ref{itm2:1} and \ref{itm2:2}, $I_1=I_2$. 
\end{proof}
\section{The case of at least one equality in \eqref{eqeqs}}\label{Discussion}
The following two lemmas will be used in the proof of Theorem \ref{th:1LL and L1L}.
\begin{lemma}\label{lemma:1from temp}
Let $\newA$ and $\mathbf \Phi(\newA)$ be as in Lemma \ref{lemma:main}.
Then the equality in \eqref{eq:eq1} holds if and only if $\newA$ can be factorized as 
\begin{equation}\label{eq:temp1}
\newA=[\operatorname{vec}(\mathbf W_1)\ \newC\otimes\mathbf x][\operatorname{vec}(\mathbf W_1)\ \newC\otimes\mathbf x]^H,
\end{equation} where
\begin{enumerate}[label=\textnormal{(\roman*)}] 
\item $\mathbf W_1\in\mathbb C^{\newn\times {\newK}}$ and $\mathbf x$ is a principal eigenvector of $\mathbf W_1\mathbf W_1^H$, i.e.,
\begin{equation*}
\mathbf W_1\mathbf W_1^H\mathbf x =\lambda_{max}(\mathbf W_1\mathbf W_1^H)\mathbf x,\quad \|\mathbf x\|=1;
\end{equation*}\label{itm:1}
\item the matrix  $\newC=[\newc_2\dots\newc_R]\in\mathbb C^{{\newK}\times (R-1)}$ has orthogonal columns;\label{itm:2}
\item $\newC^T\mathbf W_1^H\mathbf x=\mathbf 0$;\label{itm:3}
\item $\lambda_{max}(\mathbf W_1^H\mathbf W_1) = \lambda_{max}(\mathbf W_1^H\mathbf W_1+\newC^*\newC^T)$.\label{itm:5}
\end{enumerate}
 Moreover, if \eqref{eq:temp1} and \ref{itm:1}--\ref{itm:5} hold, then 
\begin{align}
\sigma(\sum\limits_{k=1}^{\newK}\newA_{kk})&=
\sigma(\mathbf W_1\mathbf W_1^H + \|\newC\|^2\mathbf x\mathbf x^H),\label{eq:temp2}\\
\sigma(\newA)&=\{\|\mathbf W_1\|^2,\|\newc_2\|^2,\dots,\|\newc_R\|^2,0,\dots,0\},\label{eq:temp3}\\
\sigma(\mathbf \Phi(\newA))&=\sigma(\mathbf W_1^H\mathbf W_1+\newC^*\newC^T)\label{eq:temp4},
\end{align}
where $\sigma(\cdot)$ denotes the spectrum of a matrix.
\end{lemma}
\begin{proof}
The proof essentially relies on  the proof of Lemma \ref{lemma:main} so we use the same notations and conventions as in the proof of Lemma \ref{lemma:main}. 

{\em Derivation of \eqref{eq:temp2}--\eqref{eq:temp4}.} Assume that \ref{eq:temp1} and \ref{itm:1}--\ref{itm:5} hold. 
Then 
$$
\newA=\sum\limits_{r=1}^R \operatorname{vec}(\mathbf W_r)\operatorname{vec}(\mathbf W_r)^H,\qquad \text{ where } \mathbf W_r=\mathbf x\newc_r^T\ \text{  for }\ r=2,\dots,R.
$$
 Hence
$$
\sum\limits_{k=1}^{\newK}\newA_{kk} = \sum\limits_{r=1}^R\mathbf W_r\mathbf W_r^H=
\mathbf W_1\mathbf W_1^H+\sum\limits_{r=2}^R
\mathbf x\newc_r^T\newc_r^*\mathbf x^H=
\mathbf W_1\mathbf W_1^H + \|\newC\|^2\mathbf x\mathbf x^H,
$$
which implies \eqref{eq:temp2}. By \ref{itm:2}, \ref{itm:3}, and the convention $\|x\|=1$ in \ref{itm:1}, the 
vectors $\operatorname{vec}(\mathbf W_r)$ are mutually orthogonal, which implies \eqref{eq:temp3}.
Finally, by \eqref{eq:eq4.5}, 
$$
\mathbf \Phi(\newA) = \sum\limits_{r=1}^R\mathbf W_r^T\mathbf W_r^* = \mathbf W_1^T\mathbf W_1^*+
\sum\limits_{r=1}^R\newc_r\mathbf x^T\mathbf x^*\newc_r^H=
\mathbf W_1^T\mathbf W_1^*+\newC\newC^H,
$$
which implies \eqref{eq:temp4}.

{\em Sufficiency.} By \ref{itm:1} and \eqref{eq:temp2}, 
$$
\lambda_{max}(\sum\limits_{k=1}^{\newK}\newA_{kk})=
\lambda_{max}(\mathbf W_1\mathbf W_1^H)+\|\newC\|^2.
$$
 By \ref{itm:5} and \ref{itm:2},
$$
\|\mathbf W_1\|^2\geq\lambda_{max}(\mathbf W_1^H\mathbf W_1)\geq\lambda_{max}(\newC^*\newC^T)=
\max\limits_{2\leq r\leq R}\|\newc_r\|^2.
$$ 
 Thus, by  \eqref{eq:temp3}, $\lambda_{max}(\newA)=\|\mathbf W_1\|^2$ and $\operatorname{tr}(\newA)=\|\mathbf W_1\|^2 + \|\newC\|^2$.
By \ref{itm:5} and \eqref{eq:temp4}, 
$\lambda_{max} (\mathbf \Phi(\newA))= \lambda_{max}(\mathbf W_1^H\mathbf W_1)$. Thus,
the left- and right-hand sides of \eqref{eq:eq1} are equal to $\lambda_{max}(\mathbf W_1\mathbf W_1^H)+
\|\mathbf W_1\|^2 +\|\newC\|^2$.

{\em Necessity.} It is clear that
the equality in \eqref{eq:eq1} holds if and only it holds in \eqref{eq:eq5} and \eqref{equation18}. 
So we  replace the inequality signs  in \eqref{eq:eq5} and \eqref{equation18} with an equality sign.

From the first line of \eqref{equation18} it follows that $\mathbf x$ satisfies \ref{itm:1}.
By the Cauchy inequality, the equality 
$$
\sum\limits_{k=1}^{\newK}\sum\limits_{r=2}^R|(\mathbf w_{kr},\mathbf x)|^2=
\sum\limits_{r=2}^R\|\mathbf w_r\|^2
$$
in \eqref{eq:eq5} would imply  that
$$
\mathbf w_{kr}=c_{kr}\mathbf x,\qquad\qquad k=1,\dots,{\newK},\quad r=2,\dots,R.
$$
for some  $c_{kr}\in\mathbb C$. Hence,
\begin{equation}\label{eq:temp6}
\mathbf w_r=[\mathbf w_{1r}^T \ \dots\ \mathbf w_{{\newK} r}^T]^T=[c_{1r}\ \dots \ c_{{\newK} r}]^T\otimes \mathbf x=\newc_r\otimes\mathbf x,\qquad r=2,\dots,R.
\end{equation}
Since $\newA=\sum\limits_{r=1}^R\mathbf w_{r}\mathbf w_{r}^H$, it follows that
$$
\newA=[\mathbf w_1\ \dots\ \mathbf w_R][\mathbf w_1\ \dots\ \mathbf w_R]^H=
[\mathbf w_1\ \newc_2\otimes\mathbf x\ \dots\ \newc_R\otimes\mathbf x]
[\mathbf w_1\ \newc_2\otimes\mathbf x\ \dots\ \newc_R\otimes\mathbf x]^H,
$$
which coincides with \eqref{eq:temp1}.
The mutual orthogonality of $\mathbf w_2,\dots,\mathbf w_R$ and the orthogonality of $\mathbf w_1$ to $\mathbf w_2,\dots,\mathbf w_R$ implies \ref{itm:2} and \ref{itm:3}, respectively.
By \eqref{eq:temp6}, 
$\mathbf W_r=\mathbf x\newc_r^T$ for $r=2,\dots,R$. Hence, the equality 
$$
\lambda_{max}\left(
\mathbf W_1^H\mathbf W_1\right)= \lambda_{max}\left(\sum\limits_{r=1}^R\mathbf W_r^H\mathbf W_r\right)
$$
in \eqref{equation18} would imply \ref{itm:5}:
$$
\lambda_{max}\left(
\mathbf W_1^H\mathbf W_1\right)=
\lambda_{max}\left(\mathbf W_1^H\mathbf W_1 + \sum\limits_{r=1}^R\newc_r^*\mathbf x^H\mathbf x\newc_r^T\right)=
\lambda_{max}(\mathbf W_1^H\mathbf W_1+\newC^*\newC^T).
$$
\end{proof}
\begin{lemma}\label{lemma:10}\qquad\\
\textup{(i)}\ 
Let $\mathbf W_1$, $\newC$,  and $\mathbf x$ satisfy conditions \ref{itm:1}--\ref{itm:5} of Lemma  \ref{lemma:1from temp}, 
$\newA$ be defined as in  \eqref{eq:temp1} and $L:=\lambda_{max}(\mathbf W_1^H\mathbf W_1)$. Then there exist $({\newK}-1)\times({\newK}-1)$ positive semidefinite matrices
$\mathbf A$ and $\mathbf B$ such that 
\begin{equation}\label{eq:ABL}
\operatorname{rank}(\mathbf A)\leq \min(\newn,{\newK})-1, \qquad
\operatorname{rank}(\mathbf B)=R-1,\qquad L\geq\lambda_{max}(\mathbf A+\mathbf B)
\end{equation}
 and
\begin{align}
\sigma(\sum\limits_{k=1}^{\newK}\newA_{kk})&=
\{L + \operatorname{tr}(\mathbf B),\lambda_1(\mathbf A),\dots,
\lambda_{\min(\newn,{\newK})-1}(\mathbf A),\underbrace{0,\dots,0}_{\newn-{\newK}}\},\label{eq:temp2AB}\\
\sigma(\newA)&=
\{L + \operatorname{tr}(\mathbf A),\lambda_1(\mathbf B),\dots,
\lambda_{R-1}(\mathbf B),\underbrace{0,\dots,0}_{\newn {\newK}-R}\},\label{eq:temp3AB}\\
\sigma(\mathbf \Phi(\newA))&=\{L\}\cup\sigma(\mathbf A+\mathbf B).\label{eq:temp4AB}
\end{align}
\textup{(ii)}\ Let a positive value $L$ and $({\newK}-1)\times({\newK}-1)$ positive semidefinite matrices $\mathbf A$ and $\mathbf B$ satisfy \eqref{eq:ABL}. Then there exists a matrix $\newA$ of form \eqref{eq:temp1} such that
 \eqref{eq:temp2AB}--\eqref{eq:temp4AB} hold.
\end{lemma}
\begin{proof}
\textup{(i)}\ 
Let $\mathbf p$ be a principal eigenvector of $\mathbf W_1\mathbf W_1^H$, i.e., $\mathbf W_1^H\mathbf W_1\mathbf p= L\mathbf p$,  $\|\mathbf p\|=1$. Then, by \ref{itm:5}, $\newC^*\newC^T\mathbf p=0$.
Let $\mathbf U_{\mathbf p}$ be an ${\newK}\times {\newK}$ unitary matrix  whose first column is $\mathbf p$. Then
\begin{equation}\label{eq:AandB}
\mathbf U_{\mathbf p}^H\mathbf W_1^H\mathbf W_1\mathbf U_{\mathbf p}=
\begin{bmatrix}
L&\mathbf 0\\
\mathbf 0& \mathbf A
\end{bmatrix},
\qquad
\mathbf U_{\mathbf p}^H\newC^*\newC^T\mathbf U_{\mathbf p}=
\begin{bmatrix}
0&\mathbf 0\\
\mathbf 0& \mathbf B
\end{bmatrix},
\end{equation}
where $\mathbf A$ and $\mathbf B$ are $({\newK}-1)\times ({\newK}-1)$ positive semidefinite matrices. It is clear that
\begin{align}
\lambda_k(\mathbf A)=\lambda_{k+1}(\mathbf W_1^H\mathbf W_1),\qquad k=1,\dots,{\newK}-1,\label{eq:lambdaA}\\
\lambda_k(\mathbf B)=\lambda_{k}(\newC^*\newC^T)=
\begin{cases}
\|\newc_{k+1}\|^2,& k=1,\dots,R-1,\\
0,& k=R,\dots,{\newK}-1.
\end{cases}\label{eq:lambdaB}
\end{align}
Now, \eqref{eq:temp3AB} follows from \eqref{eq:temp3} and \eqref{eq:AandB} and \eqref{eq:temp4AB} follows from \eqref{eq:temp4} and \eqref{eq:AandB}. To prove \eqref{eq:temp2AB} we rewrite \eqref{eq:temp2} as
\begin{equation}\label{eq:temp2ABnew}
\sigma(\sum\limits_{k=1}^{\newK}\newA_{kk})=\{\lambda_{max}(\mathbf W_1\mathbf W_1^H)+\|\newC\|^2,
\lambda_2(\mathbf W_1\mathbf W_1^H),\dots, \lambda_{\newn}(\mathbf W_1\mathbf W_1^H)\}.
\end{equation}
Since the nonzero eigenvalues of $\mathbf W_1\mathbf W_1^H$
coincide with those of $\mathbf W_1^H\mathbf W_1$ and $\|\newC\|^2=\operatorname{tr}(\mathbf B)$ it follows that \eqref{eq:temp2ABnew} is equivalent to  \eqref{eq:temp2AB}. 

\textup{(ii)}\
Let $\mathbf H$ be defined as in \eqref{eq:temp1}, where  
$$
\mathbf W_1= 
\begin{bmatrix}
\sqrt{L}&\mathbf 0\\
\mathbf 0& \widetilde{\mathbf W}_1
\end{bmatrix}, \qquad
 \newC=\mathbf U\mathbf S^{\frac{1}{2}},\qquad
\mathbf x=[1\ 0\ \dots\ 0]^T,
$$
$\widetilde{\mathbf W}_1$ is an $(\newn-1)\times (\newK-1)$ matrix such that
$\widetilde{\mathbf W}_1^H\widetilde{\mathbf W}_1=\mathbf A$ and $\mathbf U\mathbf S\mathbf U^H$ is the reduced singular value decomposition of $\begin{bmatrix}
0&\mathbf 0\\
\mathbf 0& \mathbf B
\end{bmatrix}$. 
One can easily verify   conditions \ref{itm:1}--\ref{itm:5} in Lemma \ref{lemma:1from temp} hold. Hence,
by Lemma \ref{lemma:1from temp}, \eqref{eq:temp2}--\eqref{eq:temp4} also hold. Substituting
$\mathbf W_1$, $\newC$, and $\mathbf x$ in \eqref{eq:temp2}--\eqref{eq:temp4} we obtain
\eqref{eq:temp2AB}--\eqref{eq:temp4AB}.
\end{proof}
\begin{proof}[Proof of Theorem \ref{th:1LL and L1Lnew}]
Let $\newA = \unf{T}{2}^T\unf{T}{2}^*$. Since $\sigma_1^2+\sigma_2^2=\|\mathcal T\|^2+\sigma_3^2$,  it follows that equality \eqref{eq:eq1} holds. Hence, by Lemma \ref{lemma:1from temp}, $\newA$ can be factorized as in \eqref{eq:temp1}. Therefore, there exists an  $I_2\times R$ matrix $\mathbf V$
whose columns are orthonormal and  such that $\unf{T}{2}^T = [\operatorname{vec}(\mathbf W_1)\ \newC\otimes\mathbf x]\mathbf V^H$, or equivalently,
$$
\mathbf T_k = [\mathbf w_{1k}\ \mathbf x[g_{k1}\ \dots\ g_{kR}]]\mathbf V^H,\qquad k=1,\dots,I_3.
$$
Let $\mathcal W$ and $\mathcal G$ denote the $I_1\times I_2\times I_3$ tensors whose $k$th frontal slice is
$[\mathbf w_{1k}\ \mathbf 0\ \dots\ \mathbf 0]\mathbf V^H$ and $[\mathbf 0\ \mathbf x[g_{k1}\ \dots\ g_{kR}]]\mathbf V^H$, respectively. It is clear that $\mathcal T=\mathcal W+\mathcal G$, $\mathcal W$ is  
 ML rank-$(L_1,1,L_1)$ tensor, and $\mathcal G$ is  ML rank-$(1,L_2,L_2)$ tensors, where 
$L_1\leq\min (I_1,I_3)$ and $L_2\leq\min (I_2-1,I_3)$.
\end{proof}
\begin{proof}[Proof of Theorem \ref{th:1LL and L1L}]
Let $\newA = \unf{T}{2}^T\unf{T}{2}^*$. Then
\begin{align}
&\sigma_{11}^2\geq \sigma_{12}^2\geq\dots\geq\sigma_{1I_1}^2\geq 0,& 
&\text{ are the eigenvalues of }\sum\limits_{i=1}^{I_3} {\newA}_{ii}=\unf{T}{1}\unf{T}{1}^H,&\label{eq:50} \\
&\sigma_{21}^2\geq \sigma_{22}^2\geq\dots\geq\sigma_{2I_2}^2\geq 0,&
&\text{ are the first }I_2\text{ eigenvalues of }\newA,&\label{eq:51}\\
&\sigma_{31}^2\geq \sigma_{32}^2\geq\dots\geq\sigma_{3I_3}^2\geq 0,&
&\text{ are the eigenvalues of }\mathbf\Phi(\newA)=\unf{T}{3}\unf{T}{3}^H.\label{eq:52}&
\end{align}

{\em Necessity.} By Lemmas \ref{lemma:1from temp} and \ref{lemma:10}\textup{(i)}, there exist $({\newK}-1)\times({\newK}-1)$ positive semidefinite matrices $\mathbf A$ and $\mathbf B$ such that \eqref{eq:ABL}--\eqref{eq:temp4AB} hold. Thus, by \eqref{eq:seventeen} and \eqref{eq:50}--\eqref{eq:52},
the values $\alpha_i$, $\beta_i$, and $\gamma_i$ are eigenvalues of $\mathbf A$, $\mathbf B$, and $\mathbf A+\mathbf B$, respectively. Hence, by Horn's conjecture, \eqref{eq:traceiequality} and \eqref{eq:listofinequalities} hold.

{\em Sufficiency.} Since \eqref{eq:traceiequality} and \eqref{eq:listofinequalities} hold, from  Horn's conjecture it follows that there exist $({\newK}-1)\times({\newK}-1)$ positive semidefinite matrices $\mathbf A$ and $\mathbf B$ such that $\alpha_i$, $\beta_i$, and $\gamma_i$ are eigenvalues of $\mathbf A$, $\mathbf B$, and $\mathbf A+\mathbf B$, respectively. Hence, by Lemma \ref{lemma:10}\textup{(ii)}, there exists a matrix $\newA$ of form \eqref{eq:temp1} such that  \eqref{eq:temp2AB}--\eqref{eq:temp4AB} hold. By \eqref{eq:seventeen} and 
\eqref{eq:temp3AB}, $\operatorname{rank}{\newA}\leq 1+R\leq I_2$. Let $\mathbf V$ be  an $I_2\times R$ matrix
whose columns are orthonormal and let  $\mathcal T$ denote an $I_1\times I_2\times I_3$ tensor with mode-$2$ matrix unfolding $\unf{T}{2} = \mathbf V^*[\operatorname{vec}(\mathbf W_1)\ \newC\otimes\mathbf x]^T$. Then,
$\newA = \unf{T}{2}^T\unf{T}{2}^*$. The proof now follows from \eqref{eq:50}--\eqref{eq:52}.
\end{proof}
\section{Conclusion}
In the paper we  studied  geometrical properties of the set
\begin{equation*}
\begin{split}
&\Sigma_{I_1,I_2,I_3} :=\left\{(\sigma_{11}^2,\dots,\sigma_{1I_1}^2,\sigma_{21}^2,\dots,\sigma_{2I_2}^2,\sigma_{31}^2,\dots,\sigma_{3I_3}^2):\ \right.\\
&\left.
\sigma_{nk}\ \text{is the }k\text{th largest mode-}n \text{ singular value of an } I_1\times I_2\times I_3\ \text{norm-1 tensor } \mathcal T\right\},
\end{split}
\end{equation*}
 where for each $n=1,2,3$ the values $\sigma_{nk}$ are sorted in descending order.

Let  $\pi$ denote a projection of $\mathbb R^{I_1+I_2+I_3}$ onto the
first, $(I_1+1)$th, and $(I_1+I_2+1)$th coordinates.
We have shown that there exist two convex polyhedrons of positive volume such that the set $\pi(\Sigma_{I_1,I_2,I_3})\subset\mathbb R^3$ contains one polyhedron (Theorem \ref{th:2nonsquare}) and is contained in another (Theorem \ref{th:main}). We have also shown that both
polyhedrons coincide for cubic tensors, i.e., for $I_1=I_2=I_3$ (Corollary \ref{th:2}), and can be different in the non-cubic case (Example \ref{ex:6} and Theorem \ref{th:235}).

In Theorem \ref{th:1LL and L1L}, we considered the case where the largest ML singular values of $\mathcal T$ satisfy the equality
$$
\sigma_{11}^2+\sigma_{21}^2 =  1 + \sigma_{31}^2 \text{ or }
\sigma_{11}^2+\sigma_{31}^2 =  1 + \sigma_{21}^2 \text{ or }
\sigma_{21}^2+\sigma_{31}^2 =  1 + \sigma_{11}^2
$$  
and described the preimage  $\pi^{-1}(\Sigma_{I_1,I_2,I_3})$. 
The description implies that $\pi^{-1}(\Sigma_{I_1,I_2,I_3})$ is a convex polyhedron. This seems to indicate that the whole set  
$\Sigma_{I_1,I_2,I_3}$ is also a convex polyhedron.  
As the description of $\pi^{-1}(\Sigma_{I_1,I_2,I_3})$ relies on a problem concerning the  eigenvalues of the sum of two Hermitian matrices that has long been standing, the  complete  description of $\Sigma_{I_1,I_2,I_3}$ could be an even harder problem.

We have also proved a higher-order generalizations  of Theorem  \ref{th:main} (Theorem \ref{th:main2})  and Corollary \ref{th:2} (Theorem \ref{th:2Nthorder}).
\section*{Acknowledgments}
The authors express their gratitude to the mathoverflow.net user with nickname  @fedja for his help in proving Lemma \ref{lemma:main} \cite{fedja}. 
\appendix
\section{Definition of $T_r^n$}\label{sec:appendixA} 
In our presentation we follow \cite[p. 302]{Bhatia2001}.

The set $T_r^n$ of triplets $(I,J,K)$ of cardinality $r$ can be described by induction on $r$ as follows.

Let us write $I=\{i_1<i_2<\dots<i_r\}$ and likewise for $J$ and $K$. Then for $r=1$, $(I,J,K)$
is in $T_1^n$ if $k_1=i_1+j_1-1$. For $r>1$, $(I,J,K)$ is in $T_r^n$ if
$$
\sum\limits_{i\in I}i+
\sum\limits_{j\in J}j=
\sum\limits_{k\in K}k+\frac{r(r+1)}{2},
$$
and, for all $1\leq p\leq r-1$ and all $(U,V,W)\in T_p^r$,
$$
\sum\limits_{u\in U}i_u+
\sum\limits_{v\in V}j_v=
\sum\limits_{w\in W}k_w+\frac{p(p+1)}{2}.
$$
Thus, $T_r^n$ is defined recursively in terms of $T_1^r,\dots, T_{r-1}^r$.
\bibliographystyle{siamplain}
\bibliography{refs_s_values}
\end{document}